\documentclass[10pt, twoside]{amsart}
\usepackage{amsmath,mathtools}
\usepackage{amsmath, amsthm, amssymb, amsfonts, enumerate}
\usepackage[colorlinks=true,linkcolor=blue,urlcolor=blue]{hyperref}
\usepackage{dsfont}
\usepackage{color}
\usepackage{geometry}
\usepackage{todonotes}
\usepackage{epstopdf}
\usepackage{bbm}
\usepackage{geometry}
\usepackage[utf8]{inputenc}
\usepackage{bm}
\usepackage{amsfonts}
\usepackage{amsfonts}
\usepackage{textcomp}
\usepackage{amssymb}
\usepackage{float}
\usepackage{tikz}
\usepackage{epsfig}
\usepackage{amsmath}
\usepackage[english]{babel}
\usepackage{a4}
\usepackage{enumerate}
\usepackage{tcolorbox}
\usepackage{soul}
\newcommand{\stkout}[1]{\ifmmode\text{\sout{\ensuremath{#1}}}\else\sout{#1}\fi}

\newtheorem{theorem}{Theorem}[section]

\newtheorem{remark}[theorem]{Remark}
\newtheorem{hypothesis}[theorem]{Assumption}
\newtheorem{lemma}[theorem]{Lemma}
\newtheorem{prop}[theorem]{Proposition}

\newtheorem{definition}[theorem]{Definition}

\newtheorem{example}[theorem]{Example}

\newcommand{\cp}[2]{\langle#1,#2\rangle}
\newcommand{\ds}{\displaystyle}
\newcommand{\tr}{\mathop{\mathrm{Tr}}\nolimits}

\def \R{\mathbb{R}}

\def\erre{\mathbb{R}}

\definecolor{red}{rgb}{1.0,0.0,0.0}

\definecolor{blu}{rgb}{0.0,0.0,1.0}

\definecolor{gre}{rgb}{0.03,0.50,0.03}

\usepackage[T1]{fontenc}

\title[Optimal control of stochastic delay differential equations]{Optimal control of stochastic delay differential equations and applications to path-dependent financial and economic models}

\date{\today}
\author[De Feo]{Filippo De Feo}
\address{F.~De Feo: Department of Mathematics, Politecnico di Milano, Piazza Leonardo da Vinci 32, 20133 Milano, Italy}
\email{\href{mailto:filippo.defeo@polimi.it}{filippo.defeo@polimi.it}}

\author[Federico]{Salvatore Federico}
\address{S.~Federico: Dipartimento di Economia, Universit\`a di Genova, Via F.\ Vivaldi 5, 16126, Genova, Italy}
\email{\href{mailto:salvatore.federico@unige.it}{salvatore.federico@unige.it}}

\author[\'{S}wi\k{e}ch]{Andrzej \'{S}wi\k{e}ch}
\address{A. \'{S}wi\k{e}ch: School of Mathematics, Georgia Institute of Technology, 686 Cherry Street, Atlanta, GA 30332, USA}
\email{\href{mailto:swiech@math.gatech.edu}{swiech@math.gatech.edu}}

\numberwithin{equation}{section}

\begin{document}

\begin{abstract}
In this manuscript we consider a class  optimal control problem for stochastic differential delay equations. First, we rewrite the problem in a suitable infinite-dimensional Hilbert space. Then, using the dynamic programming approach, we characterize the value function of the problem as the unique viscosity solution of the associated infinite-dimensional Hamilton-Jacobi-Bellman equation.  Finally, we prove a $C^{1,\alpha}$-partial regularity of the value function. We apply these results to path dependent financial and economic problems (Merton-like portfolio problem and optimal advertising).
\end{abstract}

\maketitle

\section{Introduction}
In this paper, we consider a class of stochastic optimal control problems with infinite horizon with delays in the state equation. Precisely, the state equation is a stochastic delay differential equation (SDDE) in $\R^{n}$ of the form
\begin{equation*}
dy(t) = \ds b_0 \left ( y(t),\int_{-d}^0 a_1(\xi)y(t+\xi)\,d\xi ,u(t) \right)
         dt 
\ds  + \sigma_0 \left (y(t),\int_{-d}^0 a_2(\xi)y(t+\xi)\,d\xi ,\, u(t) \right)\, dW(t),
\end{equation*}
with initial data $y(0)=x_{0}$ and $y(\xi)=x_1(\xi)$ for $\xi\in[-d,0]$. Here, $u$ is a control process ranging in a suitable set of admissible processes $\mathcal{U}$, and the
the goal is to minimize, for $u(\cdot)\in\mathcal{U}$,
a  functional  of the form
\begin{equation*}
J(x;u(\cdot)) =
\mathbb E\left[\int_0^{\infty} e^{-\rho t}
l(y(t),u(t)) dt \right].
\end{equation*}

 Our goal is to employ the dynamic programming approach and characterize the value function $V$ for the problem as the unique solution of the Hamilton-Jacobi-Bellman (HJB) equation in an appropriate sense, and prove its suitable regularity properties, having in mind construction of optimal feedback controls. As it is well known, the main difficulty for delay problems is in the lack of Markovianity, which prevents a direct application of the dynamic  programming method. In fact, even though the dynamic programming principle can be proved, see \cite{larssen}, it is  not immediately clear how to derive an HJB equation, which is, in general, intrinsically infinite-dimensional, as the initial datum $x_1$ is a function. If the delay kernels $a_1,a_{2}$ have a special structure, the HJB equation can be  reduced to a finite-dimensional one (see, e.g.,  \cite{larssen2}). However, this is not the case in general and other approaches are needed to tackle the problem. A possible method consists in developing and  using an It\^o's formula based on differential calculus for equations with delay, leading to a theory of the so called 
 path-dependent PDEs (see, e.g., \cite{bayraktar_keller_2018, bayraktar_keller_2022, cosso_federico_gozzi_rosestolato_touzi, ekren_2014, ekren_2016, ekren_2016b, ren_touzi_zhang, ren_rosestolato} and the references therein). Another approach, which is the one we follow here, is to lift the state equation to an infinite-dimensional Banach or Hilbert space (depending on the regularity of the data), in order to regain Markovianity. This is done at a cost of of moving to infinite-dimension.\footnote{For the procedure of  rewriting deterministic delay differential equations, we refer the readers to \cite[Part II, Chapter 4]{delfour}. For the stochastic case, one may consult \cite{choj78, DZ96, gozzi_marinelli_2004, gozzi_marinelli_savin_2006, tankov_federico} for the Hilbert case and
\cite{mohammed_1984, mohammed_1998, flandoli_zanco} for the Banach case. A ``mixed'' approach is employed in  \cite{federico_2011}. } 
To be more precise, the state equation and the cost functional are then rewritten in a suitable infinite dimensional space as
 \begin{equation*}
dY(t) = [A Y(t)+b(Y(t),u(t))]dt + \sigma(Y(t),u(t))\,dW(t), \ \ \ \ Y(0) = x=(x_{0},x_1),
\end{equation*}
and 
\begin{equation*}
J(x;u(\cdot)) =
\mathbb E \left[ \int_0^{\infty} e^{-\rho t}L(Y(t),u(t))dt \right],
\end{equation*}
 with suitable $A,b,\sigma,L$. This is explained in Section \ref{sec:infinite_dimensional_framework}.
Once this is done, one may try to employ techniques of stochastic optimal control in infinite dimensional spaces and study the associated infinite dimensional HJB equation. 
We approach this infinite dimensional HJB equation by means of viscosity solutions, whose theory is developed better in Hilbert spaces (see, e.g.,  \cite{CL1991, LY} for first order equations and deterministic problems;  \cite{fgs_book} for second order equations and stochastic problems). Thus we take the data allowing to rewrite the state equation in the Hilbert space $X:=\mathbb R^n \times L^2([-d,0];\mathbb R^n)$.
The HJB equation on $X$ has the form
\begin{equation*}
\rho v(x) - \cp{A x}{Dv(x)} + \tilde{H}(x,D_{x_{0}}v(x),D^2_{x_{0}}v(x))=0,
\quad x=(x_0,x_1(\cdot)) \in X,
\end{equation*}
where the Hamiltonian $\tilde{H}$ only involves the derivatives with respect to the finite dimensional component $x_{0}$ (see Section \ref{sec:viscosity}).

PDEs in Hilbert  spaces have been studied following at least four different approaches, based on various notions of solutions. We recall them here, together with their variants, and then describe how these approaches were applied to stochastic optimal control problems coming from delay problems.
\begin{enumerate}[i)]
\item Classical solutions (e.g. see \cite[Chapter 2]{fgs_book}). Classical solutions are rare as the regularity required for this notion of solution is typically hard to obtain.

\item Viscosity solutions (see, e.g.,   \cite[Chapter 6]{LY} and \cite[Chapter 3]{fgs_book} for a general overview, respectively, in the deterministic and the stochastic case). This approach is particularly suitable to treat first and second order fully nonlinear degenerate HJB equations. 
\item Mild solutions  in spaces of continuous functions via fixed point methods. This method was initiated in the deterministic case in \cite{barbu_daprato_1983} and then developed by many other authors (see \cite[Chapter 4]{fgs_book} for an overview).
\item Mild solutions in $(L^2,\mu)$ spaces, where $\mu$ is an invariant measure of the uncontrolled system (see, e.g., \cite[Chapter 5]{fgs_book} and \cite{goldys_gozzi_2006}).
\item Mild solutions in spaces of continuous functions via backward stochastic differential equations (BSDEs) methods (see, e.g. \cite[Chapter 6]{fgs_book} for a complete picture and also the original works \cite{fuhrman_tessitore_2002} and \cite{fuhrman_tessitore_2004}). This method relies on an extension to infinite dimension of the celebrated BSDEs approach to semi-linear HJB equation initiated by \cite{pardoux_peng_1990} and then developed by many other authors.
\item Regular solutions via convex regularization procedures. This approach was introduced mostly to study the parabolic case and requires a strong regularity of the data (e.g. the initial condition of the parabolic HJB must be convex and $C^2$), see \cite{barbu_daprato_1983}. It was then developed further only for first order equations (e.g. see \cite[Chapter 4, Bibliographic Notes]{fgs_book}).
\item Explicit (classical) solutions. This method is applied only in special cases, typically for linear-quadratic and linear-power problems; still it may provide interesting applications to economic theory. For an overview see \cite[Chapter 4, Section 10]{fgs_book}.
\end{enumerate}

Regarding applications of these various approaches to stochastic optimal control problems arising from delay equations, the general ``state of the art'' is the following.

\begin{itemize}
\item 
The approach based on mild solutions was successfully employed using the three methods (iii)--(v) to treat such problems.  However, the main drawback of these methods is that they work only for semi-linear HJB equations (i.e., when there is no control in the diffusion coefficient) and they also require many technical assumptions on the data. Smoothing properties of the transition semigroup associated to the linear part of the equation are needed for the fixed point approach\footnote{This property intrinsically does not hold for infinite dimensional systems coming from delay equations; the problem can be circumvented by looking at \emph{partial} smoothing properties, see \cite{gozzi_masiero_2017, gozzi_masiero_2017_b, gozzi_masiero_2022}, and \cite{masiero_2022}.}. The linearity of the state equation, the so called \emph{structure condition}, that is the requirement that the range of the control  operator is contained in the range of the noise, as well as constraints on the data guaranteeing the existence of an invariant measure are needed for the $(L^{2},\mu)$\footnote{See \cite{gozzi_marinelli_2004}.} approach. A special structure condition is also needed for the BSDEs approach\footnote{See, e.g., \cite{fuhrman_tessitore_masiero_2010} and also \cite[Chapter 6, Section 6.6]{fgs_book}.}. In all cases there are also some limits on the generality of the coefficients of the delay state equation, leaving out some interesting cases arising in applications; for instance, portfolio  problems, where the control naturally acts in the diffusion, leading to a fully nonlinear HJB equation.
\item Regarding the approach based on (vi), such methods were employed in \cite{faggian_2005, faggian_2008, faggian_gozzi_2010} to study a deterministic optimal investment problem with vintage capital\footnote{Even if, strictly speaking, these are not optimal control problems with delay, their  infinite dimensional formulation shares the same features with the latter, as the unbounded operator is the first derivative.}.
\item The approach based on (vii) was employed to study deterministic and stochastic problems with differential delay equations, for example in \cite{fabbri_gozzi_2008, bambi_fabbri_gozzi (2012), bambi_digirolami_federico_gozzi (2017), biffis_gozzi_prosdocini_2020, biagini_gozzi_zanella_2022, djehiche_gozzi_zanco_zanella_2022}.
\item 
The theory of viscosity solutions (ii) was first applied to
deterministic control problems. The notion of the so-called $B$-continuous viscosity solutions in infinite dimension from \cite{CL1991} is employed in  \cite{goldys_1, goldys_2} and \cite{tacconi}, where a class optimal control problem with delays and state constraints was considered: the value function was proved to be a viscosity solution of the HJB equation and a partial $C^1$-regularity of the value function was obtained. This regularity result allowed to construct optimal feedback controls. 
Another concept of viscosity solution was used in \cite{carlier_tahraoui_2010}, where the value function was characterized as the unique  solution in that sense. There have also been some results in the stochastic case.
We refer to \cite{zhou_2018, zhou_2019, zhou_2021}, where approaches using appropriately defined viscosity solutions in spaces of right-continuous functions and continuous functions were studied. In \cite{rosestolato} the authors prove existence, uniqueness and partial regularity of viscosity solutions to Kolmogorov equations\footnote{Control problems are not considered there.}  related to stochastic delay equations. Our paper can be considered as an extension of this paper to the case of fully nonlinear HJB equations.
\end{itemize}
We now describe in details the results of the paper and compare them with the related literature.
First, after rewriting the problem as an infinite dimensional control problem (see Proposition \ref{prop:equiv}), in order to apply the theory of viscosity solutions in Hilbert spaces \cite{fgs_book}, we rewrite it further by introducing a maximal dissipative operator $\tilde A$ in the state equation (see Proposition \ref{prop:Amaxdiss}) and  introduce an operator $B$ satisfying the so called weak $B$-condition (see Proposition \ref{prop:weak_b}). Then, we prove that the data of the problem satisfy some regularity conditions with respect to the norm induced by the operator $B^{1/2}$ (see Lemma \ref{lemma:regularity_coefficients_theory}). This enables us to   characterize the value function of the problem as the unique viscosity solution of the infinite-dimensional HJB equation (our first main result, Theorem \ref{th:existence_uniqueness_viscosity_infinite}). To the best of our knowledge, this is the first existence and uniqueness result for fully nonlinear HJB equations in Hilbert spaces related to a general class of stochastic optimal control problems with delays involving controls in the diffusion coefficient. 

Unfortunately, due to the lack of a good regularity results, in general the notion of viscosity solution does not provide tools to construct optimal controls. In particular, verification theorems in the context of viscosity solutions are difficult to implement, especially in the stochastic case\footnote{To have an idea about how much the problem is tricky and delicate, the reader may look at, in finite dimension, \cite[Chapter 5]{yong_zhou},  \cite{gozzi_swiech_zhou_2005}, and \cite{gozzi_swiech_zhou_2010}; for deterministic optimal control problems in finite dimension, see also \cite{bardi_1997}.  For the infinite dimensional case, the situation is clearly even more technical:  in the deterministic setting some formulations can be found in \cite[Chapter 6]{LY}, \cite{can-fr, fgs_2008}; in the stochastic case we mention a recent paper \cite{stannat_wessels_2021}. Optimal feedback controls are constructed for a class of problems with bounded evolution in \cite{mayorga_swiech_2022}.}. 
For these reasons, obtaining regularity results (even only partial) is very important. In this respect, we have to mention that some of the aforementioned papers go exactly in this direction. In  \cite{goldys_1}, in a purely deterministic framework, using the convexity of the value function, the authors are able to prove a (finite dimensional) $C^1$-regularity of the value function with respect to the $x_{0}$-component. This result is the basis for construction of optimal feedback controls via an \emph{ad hoc} verification theorem using viscosity property of the value function, a goal obtained in \cite{goldys_2}. In our paper, we have to deal with the stochastic framework and we do not impose conditions ensuring the convexity of the value function, so the techniques of  \cite{goldys_1} to prove the desired partial regularity are not applicable. Instead, we rely on smoothing properties of the noise to obtain a similar result. This approach is inspired by the arguments of \cite{rosestolato}  for Kolmogorov equations and a finite dimensional reduction procedure which first appeared in \cite{lions-infdim1}. We prove the $C^{1,\alpha}$ partial regularity of $V(x_0,x_1)$ with respect to the $x_0$-component (see Theorem \ref{th:C1alpha}) under rather general assumptions (we only require some standard Lipschitz conditions on the data, uniform in the control variable and uniform ellipticity condition of the diffusion on bounded sets). Fixing the infinite dimensional component $x_1$, we reduce the infinite dimensional HJB equation \eqref{eq:HJB} to a uniformly elliptic second order finite dimensional PDE. Then standard elliptic regularity results give the required $C^{1,\alpha}$-regularity. However, our method is different and more efficient with respect to the one used in \cite{rosestolato}. In \cite{rosestolato}, an  approximating procedure is employed: infinite dimensional SDEs with smoothed coefficients and Yosida approximations of the unbounded operator are considered, with the corresponding value functions for a smoothed out payoff function. Then, it is proved there that the finite dimensional sections of the approximating value functions are viscosity solutions of certain linear finite dimensional parabolic equations for which $C^{1,\alpha}$-estimates hold and the result follows by passing to the limit. In our paper we simplify considerably the argument, avoiding this complex approximating procedure by using deeper results from the theory of $L^p$-viscosity solutions \cite{caffarelli_c_s,swiech1997,swiech2020}.

The $C^{1,\alpha}$ partial regularity result is interesting on its own and seems to be the first one for fully nonlinear second order HJB equations with unbounded operators in Hilbert spaces. From the point of view of the control problem, it is particularly relevant as, under some additional natural assumptions, it allows to define a possible optimal feedback control. However to prove that this control is actually optimal is not an easy task. We will address this in a future publication.
Concerning regularity results for HJB equations related to delay problems in the existing literature, we observe that  in \cite{masiero_2008} and \cite{fuhrman_tessitore_masiero_2010} the full Gateaux differentiability of the solution is obtained by means of an approach via BSDEs, assuming the differentiability of the data, and  applied to problems with delays (see also \cite[Chapter 6]{fgs_book}  in Hilbert spaces). For mild solutions in $(L^2,\mu)$ spaces, in \cite[Chapter 5]{fgs_book} a first-order regularity of the solution in a Sobolev sense  is proved (applications to delay problems are provided in  \cite[Subsection 5.6]{fgs_book}). In \cite{masiero_2022}, the Gateaux differentiability of the solution  is obtained  by means of a partial smoothing of the semigroup. We also mention that in the case of bounded HJB equations (unrelated to delay problems), some $C^{1,1}$ regularity results for viscosity solutions were obtained in \cite{bensoussan_2019, bensoussan_2020, gango_meszaros_2020, lions-infdim1, mayorga_swiech_2022}  for first and second order HJB equations in Hilbert spaces and spaces of probability measures. 

We also provide two applications of our results. First, we consider a Merton-type portfolio optimization problem with path-dependency features in the dynamics of the risky asset\footnote{See \cite{merton_1969} for the original problem formulation and  \cite{guyon_2022} for a complete exposition with a quantitative analysis on why path-dependent models are important in financial modeling.}. Merton's problem with path-dependency features in the stock price was studied in \cite{pang_2019} (see also the references therein), where an exponential structure of the delay kernel is assumed; the problem is approached by employing the methods of \cite{larssen2} to reduce the infinite dimensional HJB equation to a finite dimensional one. Other results in this direction are in \cite{biffis_gozzi_prosdocini_2020, biagini_gozzi_zanella_2022, djehiche_gozzi_zanco_zanella_2022}, where the authors consider the life-cycle optimal portfolio choice problem faced by an agent receiving labor income whose dynamic has delays, while the dynamics of the risky assets are Markovian.
As a second application, we illustrate the stochastic optimal advertising problem with delays, studied in the literature in  \cite{gozzi_marinelli_2004}.

The paper is organized as follows. In Section \ref{sec:formul} we introduce the problem and state the main assumptions.
In Section \ref{sec:infinite_dimensional_framework} we rewrite the problem in an  infinite dimensional setting.
In Section \ref{sec:estimates} we prove some preliminary estimates for solutions of the state equation and the value function.
In Section \ref{sec:viscosity} we introduce the notion of viscosity solution of the HJB equation and state a theorem about the existence and uniqueness of viscosity solutions, and characterize the value function as the unique viscosity solution.
In Section \ref{sec:regularity} we prove, under additional assumptions, a partial $C^{1, \alpha}$ regularity result for the value function.
Section \ref{sec:merton} is devoted to the two applications.

\section{The optimal control problem: Setup and assumptions}\label{sec:formul}
We denote by  $M^{m \times n}$ the space of real valued $m \times n$-matrices and we denote by $|\cdot|$ the Euclidean norm in $\mathbb{R}^{n}$ as well as the norm of elements of $M^{m \times n}$ regarded as linear operators from $\mathbb{R}^{m}$ to  $\mathbb{R}^{n}$. We will write $x \cdot y$ for the inner product in $\mathbb R^n$. We consider 
the standard Lebesgue space $L^2:=L^2([-d,0];\mathbb{R}^n)$ of square integrable functions from  $[-d,0]$ to $\erre^{n}$. We denote by  
$\langle \cdot,\cdot\rangle_{L^{2}}$ the inner product in $L^2$ and by $|\cdot|_{L^{2}}$ the norm. We also consider the standard Sobolev space $W^{1,2}:=W^{1,2}([-d,0]; \mathbb{R}^n)$ of functions in $L^{2}$ admitting weak derivative in $L^{2}$, endowed with the inner product $\langle f,g\rangle_{W^{1,2}}:= \langle f,g\rangle_{L^{2}}+\langle f',g'\rangle_{L^{2}}$ and norm $|f|_{W^{1,2}}:=(|f|^2_{L^{2}}+|f'|^2_{L^{2}})^{\frac{1}{2}}$, which render it a Hilbert space. It is well known that the space $W^{1,2}$ can be identified with the space of absolutely continuous functions from $[-d,0]$ to $\R^{n}$.

Let $\tau=(\Omega,\mathcal{F},(\mathcal{F}_t)_{t\geq 0},
\mathbb{P}, W)$ be a reference probability space, that is $(\Omega,\mathcal{F},\mathbb{P})$ is a complete probability space, $W=(W(t))_{t\ge 0}$ is a standard $\mathbb{R}^q$-valued Wiener process, $W(0)=0$, and $(\mathcal{F}_t)_{t\geq 0}$ is the augmented filtration generated by $W$. We consider the following controlled stochastic differential delay equation (SDDE) 
\begin{small}
\begin{equation}
\label{eq:SDDE}
\begin{cases}
dy(t) = \ds b_0 \left ( y(t),\int_{-d}^0 a_1(\xi)y(t+\xi)\,d\xi ,u(t) \right)
         dt 
\ds  + \sigma_0 \left (y(t),\int_{-d}^0 a_2(\xi)y(t+\xi)\,d\xi ,\, u(t) \right)\, dW(t),\\
y(0)=x_0, \quad y(\xi)=x_1(\xi)\; \quad \forall \xi\in[-d,0),
\end{cases}
\end{equation}
\end{small}
where $d>0$ is the maximum delay and:
\begin{enumerate}[(i)]
\item $x_0 \in \mathbb{R}^n$, $x_1 \in L^2([-d,0]; \mathbb{R}^n)$ are the initial conditions;
\item $b_0 \colon \mathbb{R}^n \times \mathbb{R}^h \times U \to \mathbb{R}^n$, $\sigma_0 \colon \mathbb{R}^n \times \mathbb{R}^h \times U \to M^{n\times q}$;
\item $a_{i}:[-d,0]\to M^{h \times n}$ for $i=1,2$
and if $a_{i}^{j}$ is the $j$-th row of $a_i(\cdot)$, for $j=1,...,h$, then  $a_i^{j}\in W^{1,2}$ and $a_i^{j}(-d)=0$.
\end{enumerate}
The precise assumptions on $b_0,\sigma_0$ will be given later.

\begin{remark}
The condition $a_i(-d)=0$, $i=1,2$, is technical (see Remark \ref{rem:a(-d)=0}), yet not too restrictive in applications: indeed,  one can always take a slightly larger $d$, i.e. $\tilde d= d+ \varepsilon$, and extend $a_{i}$ to an absolute continuous  $M^{h \times n}$-valued function over $[-\tilde{d},0]$ in such a way that $a_i(-\tilde d)=0$.\end{remark}

We consider the following infinite horizon optimal control problem. Given $x=(x_0,x_1)\in \mathbb{R}^{n}\times L^2$, we define a cost functional of the form
\begin{equation}
\label{eq:obj-origbis}
J(x;u(\cdot)) =
\mathbb E\left[\int_0^{\infty} e^{-\rho t}
l(y^{x,u}(t),u(t)) dt \right]
\end{equation}
where $\rho>0$ is the discount factor, 
$l \colon \mathbb{R}^n \times U \to \mathbb{R} $ is the running cost and $U \subset \mathbb{R}^{p}$. For every reference probability space $\tau$ we consider the set of control processes
\[
\mathcal{U}_\tau=\{u(\cdot): \Omega\times [0,+\infty )\to U: \	 u(\cdot) \ \mbox{is} \ (\mathcal{F}_t)\mbox{-progressively measurable} \}.
\] 
We define $$\mathcal{U}=\cup_{\tau}\mathcal{U}_\tau,$$
where the union is taken over all reference probability spaces $\tau$. The goal is to minimize $J(x,u(\cdot))$ over all $u(\cdot)\in \mathcal{U}$.
This is a standard setup of a stochastic optimal control problem (see \cite{yong_zhou, fgs_book}) used to apply the dynamic programming approach. We remark (see e.g. \cite{fgs_book}, Section 2.3.2) that 
\[
\inf_{u(\cdot)\in \mathcal{U}} J(x,u(\cdot))=\inf_{u(\cdot)\in \mathcal{U}_\tau} J(x,u(\cdot))
\]
for every reference probability space $\tau$ so the optimal control problem is in fact independent of the choice of a reference probability space.

\medskip
We will assume the following conditions.
\begin{hypothesis}\label{hp:state} The functions 
$b_0,\sigma_0$ are continuous and such that there exist constants $L,C>0$ such that, for every $x,x_1,x_2 \in \mathbb R^n,z,z_1,z_2 \in \mathbb{R}^{m}$ and every $u \in U$, 
\begin{align*}
& |b_0(x,z,u)| \leq C(1+|x|+|z|), \\
& |\sigma_0(x,z,u)| \leq C(1+|x|+|z|),\\
& |b_0(x_2,z_2,u)-b_0(x_1,z_1,u)| \leq L (|x_2-x_1|+|z_2-z_1|), \\
& |\sigma_0(x_2,z_2,u)-\sigma_0(x_1,z_1,u)| \leq L(|x_2-x_1|+|z_2-z_1|). 
\end{align*}
\end{hypothesis}
Under Assumption  \ref{hp:state}, by \cite[Theorem IX.2.1]{revuz}, for each initial datum $x:=(x_0,x_1)\in \mathbb{R}^{n}\times L^2([-d,0];\mathbb R^n)$ and each control $u(\cdot)\in\ \mathcal{U}$,  there exists a unique (up to indistinguishability) strong solution to \eqref{eq:SDDE} and this solution admits a version with continuous paths that we denote by  $y^{x;u}$. The proof that the assumptions of \cite[Theorem IX.2.1]{revuz} are satisfied can be found in \cite[Proposition 2.5]{tankov_federico}.

\begin{hypothesis}\label{hp:cost} $l \colon \mathbb{R}^n \times U \to \mathbb{R} $ is continuous and is such that the following hold.

\begin{itemize}
\item[(i)]
There exist constants $K,m>0$, such that
\begin{equation}\label{eq:growth_l}
|l(z,u)| \leq K(1+|z|^m) \ \ \ \forall y\in  \mathbb{R}^n, \ \forall u \in U.
\end{equation}
\item[(ii)] There exists a local modulus of continuity for $l$, uniform in $u\in U$, i.e. for each $R>0$, there exists a nondecreasing function $\omega_{R}:\mathbb{R}^{+}\to \mathbb{R}^{+}$ such that   $\lim_{r\to 0^{+}} \omega_{R}(r)=0$ and 
\begin{equation}\label{eq:regularity_l}
|l(z,u)-l(z',u)| \leq \omega_R( |z-z'|)
\end{equation}
for every $z,z' \in \mathbb{R}^n$ such that $|z|,|z'| \leq R$ and every $u \in U$.
\end{itemize}
\end{hypothesis}
\begin{flushleft}
We will later show that, suitably reformulating the state equation in an infinite dimensional framework, the cost functional is well defined and finite for a sufficiently large  discount factor $\rho>0$ (see Assumption \ref{hp:discount}). 
\end{flushleft}

Throughout the paper we will write $C>0,\omega, \omega_R$ to indicate, respectively, a constant, a modulus continuity, and a local modulus of continuity, which may change from place to place if the precise dependence on other data is not important.
\section{The equivalent infinite dimensional Markovian
  representation}\label{sec:infinite_dimensional_framework}
The optimal control problem at hand is not Markovian due to the delay. In order to regain Markovianity and approach the problem by dynamic programming, following a well-known procedure, see \cite[Part II, Chapter 4]{delfour} for deterministic delay equations and \cite{choj78}, \cite{DZ14}, \cite{goldys_1} for the stochastic case, we reformulate the state equation by lifting it to an infinite-dimensional space.

We define 
$
X := \mathbb{R}^n \times L^2
$.
An element $x\in X$ is 
a couple $x= (x_0,x_1)$, where $x_0 \in \mathbb{R}^n$, $x_{1}\in L^{2}$; sometimes, we will write $x=\begin{bmatrix}
x_0\\ x_1
\end{bmatrix}.$ 
The space $X$ is a Hilbert space when endowed
with the inner product 
\begin{align*}
\langle x,z\rangle_{X}& := x_0\cdot  z_0 + \langle x_1,z_1 \rangle_{L^2}= x_0  z_0 + \int_{-d}^0  x_1(\xi)\cdot z_1(\xi) \,d\xi, \ \ \ x=(x_{0},x_{1}), \ z=(z_{0},z_{1})\in X.
\end{align*}
The  induced norm, denoted by $|\cdot|_X$, is then
$$
|x|_{X} = \left( |x_0|^2 + \int_{-d}^0 |x_1(\xi)|_{L^{2}}^2\,d\xi\right)^{1/2}, \ \ \ x=(x_0,x_1) \in X.
$$
For $R>0$, we  denote  $$B_R:=\{x \in X: |x|_{X} < R\}, \ \ \ B_R^0:=\{x_0 \in \mathbb R^n: |x_0| < R\}, \ \ \ B_R^1:=\{x_1 \in L^2[-d,0]: |x_1|_{L^{2}} < R\},$$ 
to be the open balls of radius $R$ in $X$, $\mathbb R^n,$ and $L^2$, respectively.
We denote by 
$\mathcal{L}(X)$ the space of bounded linear operators from $X$ to $X$, endowed with the operator norm 
$|L|_{\mathcal{L}(X)}=\sup_{|x|_{X}=1} |Lx|_{X}.$
An operator $L \in \mathcal{L}(X)$ can be seen as 
$$
Lx=\begin{bmatrix}
L_{00} & L_{01}\\
L_{10} & L_{11}
\end{bmatrix}\begin{bmatrix}
x_0\\
x_1
\end{bmatrix}, \quad x=(x_0,x_1) \in X,
$$
where $L_{00} \colon \mathbb R^n \to \mathbb R^n$, $L_{01} \colon   L^2 \to \mathbb R^n$, $L_{10} \colon \mathbb R^n \to L^2$, $L_{00} \colon L^2 \to L^2$ are bounded linear operators.
Moreover, given two separable Hilbert spaces $(H, \langle \cdot , \cdot \rangle_H), (K, \langle \cdot , \cdot \rangle_K)$, we denote by $\mathcal{L}_1(H,K)$ the space of trace-class operators endowed with the norm 
$$|L|_{\mathcal{L}_1(K,H)}=\inf \left \{\sum_{i \in \mathbb N} |a_i|_K |b_i|_{H}: Lx=\sum_{i \in \mathbb N} b_i  \langle a_i,x \rangle, a_i \in K, b_i \in  H , \forall i \in \mathbb N \right \}.$$
We also denote by  $\mathcal{L}_2(H,K)$ the space of Hilbert-Schmidt operators from $H$ to $K$ endowed with the norm
$|L|_{\mathcal{L}_2(H,K)}=(\operatorname{Tr}(L^*L))^{1/2}.$ When $H=K$ we simply write $\mathcal{L}_1(H)$, $\mathcal{L}_2(H)$.
We denote by $S(H)$ the space of self-adjoint operators in $\mathcal{L}(H)$. If $Y,Z\in S(H)$, we write $Y\geq Z$ if $\langle Yx,x \rangle\leq \langle Zx,x \rangle$ for every $x \in H$.

We define the operator $A:D(A)\subset X\to X$ 
\begin{align*}
Ax= \begin{bmatrix}
0\\ x_1'
\end{bmatrix}, \quad D(A) = \left\{ x=(x_0,x_1) \in X: \ x_1 \in W^{1,2}, \ x_1(0)=x_0\right\}.
\end{align*}
By \cite[Theorem 4.2]{delfour}, the operator $A$ is  the generator of a strongly continuous semigroup $e^{tA}$ on $X$, whose explicit expression is 
$$ e^{At}   x =
\begin{bmatrix}
 x_0\\
I_{[-d,0]}(t + \cdot)  x_1(t+ \cdot) + I_{[0,\infty)} (t + \cdot) x_0
\end{bmatrix}, \quad  x= ( x_0,  x_1)\in X.$$
Notice that 
\begin{equation}\label{est_sem}
|e^{tA}|_{\mathcal{L}(X)}\leq (2(1+d))^{1/2}, \ \ \forall t\geq 0.
\end{equation}
The adjoint of the operator $A$ is  the operator (see \cite[Proposition 3.4]{goldys_1})
\begin{align}\label{eq:A*}
 A^* x = \begin{bmatrix}
 x_1(0)\\
 -x_1^\prime
\end{bmatrix}, \quad
 D(A^*) = \left\{ x=(x_0,x_1)  \in X:
       x_1 \in W^{1,2}, \ x_1(-d)=0\right\}.
\end{align}
%
\begin{flushleft}
We now define $b \colon X \times U \to X$ (with a small abuse of notation for $b_0(x,u)$) by
\end{flushleft}
\begin{align*}
b(x,u)=\begin{bmatrix}
b_0 ( x,u)\\
0
\end{bmatrix}
=
\begin{bmatrix}
b_0 \left ( x_0,\int_{-d}^0 a_1(\xi)x_1(\xi)\,d\xi ,u \right)\\
0
\end{bmatrix}, \quad x=(x_0,x_1) \in X,  \ u \in U
\end{align*}
and  $\sigma \colon X \times U \to \mathcal{L}(\mathbb{R}^q,X)$ (again with a small abuse of notation for $\sigma_0(x,u)$) by 
$$\sigma(x,u)w=\begin{bmatrix}
\sigma_0 ( x,u) w\\
0
\end{bmatrix}=
\begin{bmatrix}
\sigma_0 \left ( x_0,\int_{-d}^0 a_2(\xi)x_1(\xi)\,d\xi ,u \right)w\\
0
\end{bmatrix},\quad x=(x_0,x_1) \in X, \ u \in U, \ w \in \mathbb{R}^q.
$$
Given $x\in X$ and a control process $u(\cdot)\in\mathcal{U}$, we consider  the following infinite-dimensional stochastic differential equation:
\begin{equation}
\label{eq:abstract}
\begin{cases}
dY(t) = [A Y(t)+b(Y(t),u(t))]dt + \sigma(Y(t),u(t))\,dW(t) \\[4pt]
Y(0) = x.
\end{cases}
\end{equation}
It is well known (see e.g. \cite{DZ14}) that there exists a unique mild solution to \eqref{eq:abstract}, that is an $X$-valued progressively measurable stochastic process $Y$ satisfying
\begin{align*}
Y(t)=e^{A t}x+ \int_0^t e^{A(t-s)}b(Y(s),u(s)) ds+ \int_0^t e^{A(t-s)}\sigma(Y(s),u(s)) dW(s), \ \ \ \forall t\geq 0.
\end{align*}
The infinite dimensional stochastic differential equation \eqref{eq:abstract} is linked to \eqref{eq:SDDE} by the following result, see   \cite[Theorem 3.4]{tankov_federico} (cf. also the original result in the linear case \cite{choj78}).
\begin{prop}\label{prop:equiv} Let  Assumption \ref{hp:state} hold. Given $x\in X$ and $u(\cdot)\in\mathcal{U}$, 
let $y^{x,u}$ be the unique strong solution to \eqref{eq:SDDE} and let $Y^{x,u}$ be the unique mild solution to \eqref{eq:abstract}. Then  
$$Y^{x,u}(t)=(y^{x,u}(t),y^{x,u}(t+\cdot)|_{[-d,0]}), \ \ \ \forall t\geq 0.$$\end{prop}
  
Proposition \eqref{prop:equiv} provides a Markovian
reformulation of the optimal control problem in the Hilbert space $X$.
 Indeed, the functional \eqref{eq:obj-origbis} can be rewritten in $X$ as
\begin{equation}
\label{ex1jbisbis}
J(x;u(\cdot)) =
\mathbb E \left[ \int_0^{\infty} e^{-\rho t}[L(Y^{x,u}(t),u(t))]dt \right],
\end{equation}
where $L : X \times U\to\erre$ is defined by
$$
L(x,u) := l(x_0,u), \quad x=(x_0,x_1) \in X, \ u \in U.
$$
The value function $V$ for this problem is defined by 
$$
V(x) := \inf_{u(\cdot)\in\mathcal{U}} J(x;u(\cdot)).
$$
\subsection{Rewriting of the state equation using a maximal monotone operator}
In this subsection we rewrite the state equation \eqref{eq:abstract} using a maximal dissipative operator $\tilde A$. This is needed to be in the framework of the theory of viscosity solutions to the associated Hamilton-Jacobi-Bellman (HJB) equation (see \cite[Chapter 3]{fgs_book}). To do this we introduce the operator $\tilde A \colon D(\tilde A) \subset X \to X$ defined by
$\tilde{A}:=A- (x_0,0)$, i.e. 
\begin{equation}\label{eq:tildeA}
\tilde A x= \begin{bmatrix}
-x_0\\
x_1'
\end{bmatrix}, \quad D(\tilde A)=D(A)= \left\{ x=(x_0,x_1) \in X: x_1 \in W^{1,2}, \ x_1(0)=x_0\right\}.
\end{equation}
\begin{prop}\label{prop:Amaxdiss}
The operator $\tilde{A}$ defined in \eqref{eq:tildeA} is maximal dissipative.
\end{prop}
\begin{proof}
Let $x\in D(\tilde{A})$. Taking into account that $x_1(0)=x_0$, we have
$$\langle\tilde{A}x,x\rangle_{X}=-|x_0|^2+\int_{-d}^0\langle x'_1(\xi), x_1(\xi)\rangle d\xi=-|x_0|^2+\left[\frac{|x_1(\xi)|^2}{2}\right]^0_{-d}=- \frac{|x_0|^2}{2} -|x_1(-d)|^2\leq 0,$$
so $\tilde{A}$ is dissipative.

In order to prove that $\tilde{A}$ is maximal dissipative we need to show that $\mathcal{R}(\lambda I - \tilde{A})=X$ for some $\lambda>0$. This means that we have to show that there exists some $\lambda >0$ such that, for each $y=(y_0,y_1)\in X$, we can find $x=(x_0,x_1)\in D(\tilde{A})$ such that $\lambda x-\tilde A x = y$, i.e.
\begin{equation}\label{eq:proof_maximality}
\begin{cases}
\lambda x_0 +x_0=y_0,\\
\lambda x_1-x_1'=y_1.
\end{cases}
\end{equation}
Indeed this is true for 
each $\lambda>0$. Let $\lambda>0$  and take an arbitrary $y=(y_0,y_1)\in X$. By the first equation in \eqref{eq:proof_maximality}, we have 
$$x_0=\frac{1}{1+\lambda}y_0.$$ 
Now recall that if $x=(x_0,x_1)\in D(\tilde{A})$, we must have $x_1(0)=x_0=y_0/(1+\lambda)$. The second equation of \eqref{eq:proof_maximality} is then the ODE
$$x_1'(\xi)=\lambda x_1(\xi)-y_1(\xi) \quad \forall \xi \in [-d,0], \quad x_1(0)=y_0/(1+\lambda).$$
Its unique solution $x_1  \in W^{1,2}$ is given by
\begin{align*}
x_1(\xi)=\frac{1}{1+\lambda}e^{\lambda \xi} y_0 +\int_\xi^0e^{-\lambda r} y_1(r) dr \quad \forall \xi \in [-d,0].
\end{align*}
Therefore, we found the unique solution $x=(x_0,x_1)\in D(\tilde{A})$ to the abstract equation $(\lambda I-\tilde A) x = y$, showing the maximality of $\tilde A$.
\end{proof}
By Proposition \ref{prop:Amaxdiss}, we can now rewrite \eqref{eq:abstract} using the maximal dissipative operator $\tilde A$ as
\begin{equation}
\label{eq:abstract_dissipative_operator}
\begin{cases}
dY(t) = \left[ \tilde A Y(t)+\tilde b(Y(t),u(t)) \right]dt +  \sigma(Y(t),u(t))\,dW(t),  \\[4pt]
Y(0) = x\in X,
\end{cases}
\end{equation}
where  $\tilde b \colon X \times U \to X$ is defined by
\begin{align*}
\tilde b(x,u) & =  b(x,u) + \begin{bmatrix}
x_0\\
0
\end{bmatrix}
=
\begin{bmatrix}
b_0 \left ( x_0,\int_{-d}^0 a_1(\xi)x_1(\xi)\,d\xi ,u \right)+x_0\\
0
\end{bmatrix}, \quad x=(x_0,x_1) \in X,  \ u \in U.
\end{align*}
Since the operator $\tilde A$ is the sum of $A$ and a bounded operator, by \cite[Corollary 1.7]{engel_nagel} we conclude that \eqref{eq:abstract_dissipative_operator} is equivalent to  \eqref{eq:abstract} and they have the same (unique) mild solution given, in terms of $\tilde{A}$,  by
\begin{equation}\label{eq:mild}
Y(t) =e^{{\tilde A t}}x+\int_{0}^{t}  e^{\tilde A (t-s)}\, \tilde b(Y(s),u(s)) d s+\int_{0}^{t} e^{\tilde A(t-s)}\sigma(Y(s),u(s))d W(s).
\end{equation}
\subsection{Weak $B$-condition}\label{sec:operator_B}
In this subsection we recall the concept of weak $B$-condition (for $\tilde{A}$) and introduce an operator $B$ satisfying it. This concept is  fundamental in the theory of viscosity solutions in Hilbert spaces, see \cite[Chapter 3]{fgs_book}, which will be used in this paper.
We first notice that by \eqref{eq:A*} and the definition of $\tilde A$, the adjoint operator $\tilde{A}^{*}:D(\tilde A^*)\subset X\to X$, is given by
\begin{align*}
 \tilde A^* x = \begin{bmatrix}
 x_1(0)-x_0\\
 -x_1'
\end{bmatrix}, \quad
 D(\tilde A^*) = D( A^*)= \left\{ x=(x_0,x_1)  \in X:
       x_1 \in W^{1,2}([-d,0];\mathbb{R}^n), \ x_1(-d)=0\right\}.
\end{align*}

\begin{definition}\label{def:weakB}(See \cite[Definition 3.9]{fgs_book})
We say that $B \in \mathcal{L}(X)$ satisfies the weak $B$-condition (for $\tilde{A}$) if the following hold:
\begin{enumerate}[(i)]
\item $B$ is strictly positive, i.e. $\langle Bx,x \rangle_{X} >0$ for every $x \neq 0$;
\item $B$ is self-adjoint;
\item $\tilde{A}^* B \in \mathcal{L}(X)$;
\item There exists $C_0 \geq 0$ such that
$$\langle \tilde{A}^* Bx,x \rangle_{X} \leq C_0 \langle Bx,x \rangle_{X}, \quad \forall x \in X.$$
\end{enumerate}
\end{definition}

We construct an operator $B$ satisfying the weak $B$-condition. Let $\tilde{A}^{-1}$ be the inverse of the operator $\tilde{A}$. Its explicit expression can be derived from the proof of the second part of Proposition \ref{prop:Amaxdiss} with $y=-x$ and $\lambda=0$ (which does not invalidate the calculations in the proof) and is given by
\begin{equation}\label{eq:tilde_A_inverse}
 \tilde{A}^{-1}x=\left(-x_0, -x_0 -\int_\cdot^0 x_1(\xi)d\xi\right), \quad \forall x=(x_0,x_1)\in X.
\end{equation}
Notice that $\tilde A^{-1} \in \mathcal{L}(X)$. Moreover, since $\tilde A^{-1}$ is continuous as an operator from $X$ to $W^{1,2}$, and the embedding  $W^{1,2} \hookrightarrow L^{2}$ is compact, $\tilde A^{-1}: L^2\to L^2$ is compact.
Define now
\begin{equation}\label{def:B}
B:=(\tilde{A}^{-1})^*\tilde{A}^{-1}=(\tilde{A}^*)^{-1}\tilde{A}^{-1}\in\mathcal{L}(X).
\end{equation}
$B$ is compact by the compactness of $\tilde A^{-1}$. 
\begin{prop}\label{prop:weak_b}
The operator $B$ defined in \eqref{def:B} satisfies the weak $B$-condition for $\tilde A$ with $C_0=0$.
\end{prop}
\begin{proof}
It is immediate to see that $B\in \mathcal{L}(X)$, $\tilde A^* B= \tilde A^{-1} \in \mathcal{L}(X)$, and $B$ is self adjoint.
Moreover,
\begin{align*}
\langle Bx,x \rangle_{X}= \langle \tilde A^{-1}x,\tilde A^{-1}x \rangle_{X}= |\tilde A^{-1}x|_{X} \geq 0, \ \ \ \ \forall  x \in X.
\end{align*}

We now show that  $B$ is strictly positive. Let $x\neq 0$. If $x_0\neq  0$, we have $|\tilde A^{-1}x|_{X} > 0$. On the other hand, if $x_{0}=0$, then we must have $x_1 \neq 0$ and  then the function $ \int_\cdot^0 x_1(\xi)d\xi\not\equiv 0$.  In both cases, by \eqref{eq:tilde_A_inverse},  we deduce  the strict positivity of $B$.

Finally, by the dissipativity of $\tilde A$, we have
\begin{align*}
\langle A^* Bx,x \rangle_{X} = \langle \tilde A^{-1}x ,x \rangle_{X} = \langle y,\tilde Ay \rangle_{X} \leq 0
\end{align*}
by taking  $y= \tilde A^{-1}x$. The claim is proved.
\end{proof} 
Observe that 
\begin{align}\label{eq:representation_B}
B x= 
  \begin{bmatrix}
    B_{00} & B_{01}  \\
   B_{10} & B_{11}
  \end{bmatrix} 
   \begin{bmatrix}
   x_0 \\
  x_1
  \end{bmatrix}, \ \ \ \ x=(x_0,x_1)\in X. 
\end{align}
By the strict positivity of $B$, $B_{00}\in M^{n\times n}$ is  strictly positive   and $B_{11}$ is strictly positive as an operator $L^2\to L^2$. Moreover, since $B$ is strictly positive and self-adjoint, the operator $B^{1/2}\in \mathcal{L}(X)$ is well defined, self-adjoint and strictly positive. We introduce the $|\cdot |_{-1}$-norm on $X$ by
\begin{align}\label{eq:properties_norm_-1}
\nonumber|x|_{-1}^2&=\langle B^{1/2}x, B^{1/2}x\rangle_{X}= \langle Bx,x\rangle_{X}\\&=\langle (\tilde{A}^{-1})^*\tilde{A}^{-1}x,x\rangle_{X}=\langle \tilde{A}^{-1}x,\tilde{A}^{-1}x\rangle_{X}= |\tilde A^{-1}x|^{2}_{X} \quad \forall x \in X. 
\end{align}
We define 
%
%
$$X_{-1}:= \ \mbox{  the completion of $X$ under} \ |\cdot|_{-1},$$
which is a Hilbert space endowed with the inner product
$$
\langle x,y\rangle_{-1}:=\langle B^{1/2}x,B^{1/2}y\rangle_{X}= \langle Bx,y\rangle_{X}= \langle \tilde{A}^{-1}x,\tilde A^{-1}y\rangle_{X}.
$$
Notice that $|x|_{-1}\leq |\tilde A^{-1}|_{\mathcal L(X)}|x|_X$; in particular,  we have $(X,|\cdot|) \hookrightarrow (X_{-1},|\cdot|_{-1})$. 
Moreover,
strict positivity of $B$ ensures that the operator $B^{1 /2}$  can be extended to an isometry 
$$B^{1 /2} \colon (X_{- 1},|\cdot|_{-1}) \to (X,|\cdot|_{X}).$$

%
%
%
By \eqref{eq:properties_norm_-1} and  an application of \cite[Proposition B.1]{DZ14},  we have  
$\mbox{Range}(B^{1/2})=\mbox{Range}((\tilde A^{-1})^*)$. Since $\mbox{Range}((\tilde A^{-1})^*)=D(\tilde A^*)$, we have 
\begin{align}\label{eq:R(B^1/2)=D(A^*)}
\mbox{Range}\big(B^{1/2}\big)=D(\tilde A^*).
\end{align}
By \eqref{eq:R(B^1/2)=D(A^*)}, the operator $\tilde A^* B^{1/2}$ is well defined on the whole space $X$. Moreover, since $\tilde A^*$ is closed and $B^{1/2}\in\mathcal{L}(X)$, $\tilde A^* B^{1/2}$ is a closed operator. Thus, by the closed graph theorem, we have 
\begin{equation}\label{eq:A*B^1/2}
\tilde A^* B^{1/2} \in \mathcal{L}(X).
\end{equation}

\begin{remark}
In the infinite dimensional theory of viscosity solutions it is only required that $\tilde A^*B \in \mathcal{L}(X)$ (condition $(iii)$ of Definition \ref{def:weakB}). Such an operator can be constructed for any maximal dissipative operator $\tilde A$ (see, e.g.,  \cite[Theorem 3.11]{fgs_book}). In the case of the present paper, in addition, we also have $\tilde A^* B^{1/2} \in \mathcal{L}(X)$. Hence, here
 $B$ is better than a generic $B$ that one usually uses in the standard theory. 
\end{remark}
By \eqref{eq:tilde_A_inverse}, we immediately notice that
\begin{equation}\label{eq:|x_0|_leq_|x|}
|x_0|\leq |x|_{-1}, \quad \forall x=(x_0,x_1)\in X.
\end{equation}
Since $B$ is a compact, self-adjoint and strictly positive operator on $X$, by the spectral theorem $B$ admits a set  of  eigenvalues $\{\lambda_i \}_{i \in \mathbb{N}} \subset (0, +\infty)$ such that $\lambda_{i}\to 0^{+}$ and a corresponding set $\{f_i \}_{i \in \mathbb{N}} \subset X$  of  eigenvectors  forming an orthonormal basis of $X$. By taking $\{e_i \}_{i \in \mathbb{N}}$ defined by $e_i:=\frac{1 } {\sqrt \lambda_i} f_i$, we then get   an orthonormal basis of $X_{-1}$. 
We set
$X^N := \mbox{Span} \{f_1,...f_N\}= \mbox{Span} \{e_1,...e_N\}$ for  $N\geq 1$, and let 
$P_N \colon X \to X$
 be the orthogonal projection onto $X_N$ and 
$Q_N := I - P_N$. 
Since $\{e_i \}_{i \in \mathbb{N}}$ is an orthogonal basis of $X_{-1}$, the projections $P_N,Q_N$ extend to orthogonal projections in $X_{-1}$ and we will use the same symbols to denote them.
We notice that 
\begin{equation}\label{eq:BP_BQ}
B P_N =P_N B, \quad B Q_N =Q_N B.
\end{equation}
Therefore, since $|BQ_N|_{\mathcal L(X)}=|Q_NB|_{\mathcal L(X)}$ and $B$ is compact, we get
\begin{equation}\label{eq:norm_BQ_N_to_zero}
\lim_{N \to \infty }|BQ_N|_{\mathcal L(X)} =0.
\end{equation}

\section{Estimates for the state equation and the value function}\label{sec:estimates}
In this section we prove estimates for solutions of the state equation, the cost functional and the value function. 
We first derive regularity properties for $\tilde b, \sigma, L$.

\begin{lemma}\label{lemma:regularity_coefficients_theory} 
Let Assumptions \ref{hp:state} and \ref{hp:cost} hold.
There exists $C>0$ and a local modulus of continuity $\omega$ such that  the following hold true for every $x,y \in X, u \in U$:
\begin{align}
& |\tilde b(x, u)-\tilde b(y, u)| \leq C|x-y|_{-1} \label{eq:b_lip},  \\
& \langle \tilde b(x, u)-\tilde b(y, u), B(x-y)\rangle_{X} \leq C|x-y|_{-1}^{2},  \label{eq:b_B_property}  \\
& |\tilde b(x, u)| \leq C(1+|x|_{X}) , \label{eq:b_sublinear}  \\
& |\sigma(y,u)-\sigma(x,u)|_{\mathcal{L}_{2}(X)} \leq C |x-y|_{-1}, \label{eq:G_lipschitz_norm_B} \\
& |\sigma(x,u)|_{\mathcal{L}_{2}} \leq C (1+|x|_{X}), \label{eq:G_bounded}  \\
&|L(x, u)-L(y, u)| \leq \omega\left(|x-y|_{-1}, R\right), \label{eq:L_unif_cont_norm_B}  \\
& |L(x, u)| \leq C\left(1+|x|_{X}^{m}\right).   \label{eq:L_growth} 
\end{align}
Moreover, 
\begin{equation}\label{eq:trace_goes_to_zero}
\lim_{N \to \infty} \sup_{u \in U} \operatorname{Tr}\left[\sigma(x,u) \sigma(x,u)^{*} B Q_{N}\right]=0, \ \ \ \forall x\in X.
\end{equation}
\end{lemma}
\begin{proof}


\emph{Proof of  \eqref{eq:b_lip} and \eqref{eq:b_B_property}.} 
By our assumptions, we have  $(0,a_1^i) \in D(\tilde{A}^{*})$ for every $i \leq h$, where $a_1^i$ are the rows of the matrix $a_1$. Then, 
\begin{align}\label{eq:sublinearity_int_wrt_norm_-1}
\left|\int_{-d}^0 a_1(\xi)x_1(\xi)d\xi\right|^2&= \sum_{i=1}^h | \langle (0,a_1^i), x\rangle_{X}|^2 =\sum_{i=1}^h | \langle (0,a_1^i), \tilde{A}\tilde{A}^{-1}x\rangle_{X}|^2 \nonumber \\
 & \leq \sum_{i=1}^h | \langle \tilde{A}^{*}(0,a_1^i), \tilde{A}^{-1}x\rangle_{X}|^2 \leq \sum_{i=1}^h |\tilde{A}^{*}(0,a_1^i) |_{X}^2 |x|_{-1}^2  =  C |x|_{-1}^2.
\end{align}
Thus, by \eqref{eq:sublinearity_int_wrt_norm_-1} and \eqref{eq:|x_0|_leq_|x|}, we get 
\begin{align}\label{eq:lip_b_norm_-1}
|\tilde b(x,u)-\tilde b (y,u)|_{-1} & \leq C |\tilde b(x,u)-\tilde b (y,u)| _{X}\nonumber \\
&\leq \left | b_0 \left (x_0,\int_{-d}^0 a_1(\theta) x_1(\theta)d \theta,u \right ) -b_0 \left (y_0,\int_{-d}^0 a_1(\theta) y_1(\theta)d \theta,u \right)\right| +|x_0 -y_0 | \nonumber \\
& \leq C \left (|x_0-y_0|+\left |\int_{-d}^0 a_1(\theta) (x_1(\theta)-y_1(\theta)) d\theta \right | \right )  \leq C |x-y|_{-1}.
\end{align}
Inequality \eqref{eq:b_B_property} follows trivially from \eqref{eq:b_lip}.

\emph{Proof of  \eqref{eq:b_sublinear}.} The estimate follows easily from Assumption \ref{hp:state}, the definition of $\tilde b$ and the fact that 
$a_1^j\in W^{1,2}$ for $j=1,...,h$.


\emph{Proof of   \eqref{eq:G_lipschitz_norm_B}.} We have
\begin{align*}
|\sigma(y,u)-\sigma(x,u)|_{\mathcal{L}_{2}(X)} &  = \left | \sigma_0 \left (x_0,\int_{-d}^0 a_2(\theta) x_1(\theta)d \theta,u \right ) -\sigma_0 \left (y_0,\int_{-d}^0 a_2(\theta) y_1(\theta)d \theta,u \right)\right|. 
\end{align*}
Since $(0,a_2) \in {D}(\tilde{A}^{*})$, by \eqref{eq:sublinearity_int_wrt_norm_-1} with $a_2$ in place of $a_1$, we have
\begin{align*}
 & \left | \sigma_0 \left (x_0,\int_{-d}^0 a_2(\theta) x_1(\theta)d \theta,u \right )  -\sigma_0 \left (y_0,\int_{-d}^0 a_2(\theta) y_1(\theta)d \theta,u \right)\right|  \\
&  \leq C \left (|x_0-y_0|+\left |\int_{-d}^0 a_2(\theta) (x_1(\theta)-y_1(\theta)) d\theta \right | \right ) \leq C |x-y|_{-1}.
\end{align*}

\emph{Proof of  \eqref{eq:G_bounded}.} Inequality \eqref{eq:G_bounded} follows from the definition of $\sigma \colon X \times U \to \mathcal{L}(\mathbb{R}^q, X)$, since
\begin{align*}
|\sigma(x,u)|_{\mathcal{L}_2(X)}^2& = \sum_{i=1}^n |\sigma_0(x,u) v_i|^2  =|\sigma_0(x,u)|^2 \leq C(1+ |x|_{X}^2),
\end{align*}
where $\{v_i \}_{i =1,..., q}$ is the canonical basis of $\mathbb{R}^q$.
%

\emph{Proof of   \eqref{eq:L_unif_cont_norm_B}.} It follows from the definition of $L$, \eqref{eq:regularity_l} and \eqref{eq:|x_0|_leq_|x|}, as we have
\begin{align*}
|L(x, u)-L(y, u)|  \leq \omega_R(|x_0-y_0|) \leq \omega_R(|x-y|_{-1}).
\end{align*}

\emph{Proof of   \eqref{eq:L_growth}.} It follows from the definition of $L$ and \eqref{eq:growth_l}.

\emph{Proof of  \eqref{eq:trace_goes_to_zero}.} We notice that by \cite[Appendix B]{fgs_book} and \eqref{eq:G_bounded}, we have:
\begin{small}
 \begin{align*}
|\operatorname{Tr}\left[\sigma(x,u) \sigma(x,u)^{*} B Q_{N}\right]|
& \leq  |\sigma(x,u) \sigma(x,u)^{*} BQ_N|_{\mathcal L_1(X)} \leq  |\sigma(x,u) \sigma(x,u)^{*}|_{\mathcal L_1(X)} |B Q_N|_{\mathcal L(X)} \\
& \leq    |\sigma(x,u)|_{\mathcal L_2(X)}  |\sigma(x,u)^{*}|_{\mathcal L_2(X)}|B Q_N|_{\mathcal L(X)} \\
&=  |\sigma(x,u)|_{\mathcal L_2(X)}^2 |B Q_N|_{\mathcal L(X)}  \leq C(1+|x|^2) |B Q_N|_{\mathcal L(X)}.
\end{align*}
\end{small}
Thus we obtain \eqref{eq:trace_goes_to_zero} by taking the supremum over $u$, letting $N \to \infty$ and using \eqref{eq:norm_BQ_N_to_zero}.
\end{proof}
\begin{remark}\label{rem:a(-d)=0}
The requirements $a_1(-d)=a_2(-d)=0$ are in general necessary to get \eqref{eq:sublinearity_int_wrt_norm_-1}. Indeed, consider for example the case $a_1(\cdot)\equiv1$. In such a case the sequence
$$x^N=(x_0^N,x_1^N), \ \ \  x_0^N=0, \ x_1^N=N I_{[-d,-d+1/N]}, \ \ \ N\geq 1,$$ is such that
 $$\left|\int_{-d}^0 a_1(\xi) x_1^N(\xi)d\xi\right|=1 \ \ \forall  N \geq 1, \ \ \ \ |x^N|_{-1}\rightarrow 0 \ \mbox{when} \ N\rightarrow\infty,$$
and then \eqref{eq:sublinearity_int_wrt_norm_-1} cannot be satisfied. 
\end{remark}
We recall \cite[Proposition 3.24]{fgs_book}. Set 
\[ \rho_{0}:=
\begin{cases}
0,              & \text { if } m =0,\\
C m+\frac{1}{2} C^{2} m,  &\text { if } 0<m<2, \\
C m+\frac{1}{2} C^{2} m(m-1),        & \text { if } m \geq 2, 
\end{cases}
\]
where $C$ is the constant appearing in \eqref{eq:b_sublinear} and \eqref{eq:G_bounded}, and $m$ is the constant from Assumption \ref{hp:cost} and \eqref{eq:L_growth}.

\begin{prop}\label{prop:expect_Y(s)} (\cite[Proposition 3.24]{fgs_book})
Let Assumption \ref{hp:state} holds and let $\lambda>\rho_0$.
Let $Y(t)$ be the mild solution of \eqref{eq:abstract_dissipative_operator} with initial datum $x \in X$ and control $u(\cdot) \in \mathcal U$. Then, there exists $C_\lambda>0$ such that
$$
\mathbb{E}\left[|Y(t)|_X^m\right] \leq C_\lambda\left(1+|x|_X^m\right)e^{\lambda t}, \quad \forall  t \geq 0.
$$
\end{prop}
We need the following assumption.
\begin{hypothesis}\label{hp:discount}
$\rho >\rho_{0}.$
\end{hypothesis}

\begin{prop}\label{prop:growth_V}
Let Assumptions \ref{hp:state}, \ref{hp:cost},  and \ref{hp:discount} hold. There exists $C>0$ such that
$$
|J(x;u(\cdot))|\leq C(1+|x|_X^m) \quad \forall x \in X,  \ \forall u(\cdot) \in \mathcal{U}.
$$
Hence, 
$$|V(x)|\leq C(1+|x|_X^m), \ \ \ \forall x \in X.$$
\end{prop}
\begin{proof}
By \eqref{eq:L_growth} and Proposition \ref{prop:expect_Y(s)} applied with $\lambda=(\rho+\rho_0)/2$, we have
\[
|J(x,u(\cdot))| \leq C \int_0^\infty e^{-\rho t} \mathbb E[ (1+|Y(t)|_X)^m] dt\leq C (1+|x|_X^m), \ \ \ \forall x \in X,  \ \forall u(\cdot) \in \mathcal{U}.
\]
The estimate on $V$  follows from this.
\end{proof}
Next, we show continuity properties of $V$. We recall first the notion of $B$-continuity (see \cite[Definition 3.4]{fgs_book})
\begin{definition}\label{def:B-semicontinuity}
 Let $B \in\mathcal{L}(X)$ be a strictly positive self-adjoint operator. 
 A function $u: X \rightarrow \mathbb{R}$ is said to be $B$-upper semicontinuous (respectively, $B$-lower semicontinuous) if, for any sequence $\left\{ x_{n}\right\}_{n \in \mathbb{N}}\subset X$ such that $ x_{n} \rightharpoonup x \in X$ and $B x_{n} \rightarrow B x$ as $n \rightarrow \infty$, we have
$$
\limsup _{n \rightarrow \infty} u\left( x_{n}\right) \leq u(x) \ \ \ \mbox{
(respectively, }
 \ \liminf _{n \rightarrow \infty} u\left( x_{n}\right) \geq u(x)).$$ 
  A function $u: X \rightarrow \mathbb{R}$ is said to be $B$-continuous if it is both $B$-upper semicontinuous and $B$-lower semicontinuous.
\end{definition}
We remark that, since the operator $B$ defined in \eqref{def:B} is compact,  in our case $B$-upper/lower semicontinuity is equivalent to the weak sequential upper/lower semicontinuity, respectively.

\begin{prop}\label{prop:V_continuous_norm_-1}
Let Assumptions \ref{hp:state}, \ref{hp:cost},  and \ref{hp:discount} hold.
For every $R>0,$ there exists a  modulus of continuity $\omega_R$ such that 
\begin{equation}\label{Va-1}
|V(x)-V(y)| \leq \omega_R(|x-y|_{-1}) , \ \ \ \forall x,y \in X, \ \mbox{s.t.} \ |x|, |y| \leq R.
\end{equation}
Hence $V$ is $B$-continuous and thus weakly sequentially continuous. 
\end{prop}
\begin{proof}
We prove the estimate
$$
|J(x,u)-J(y,u)| \leq \omega_R(|x-y|_{-1}) \quad \forall x,y \in X: \  |x|, |y| \leq R, \forall u(\cdot)\in \mathcal \mathcal{U}, 
$$
as in \cite[Proposition 3.73]{fgs_book}, since the assumptions of the latter are satisfied due to Lemma \ref{lemma:regularity_coefficients_theory}. Then, \eqref{Va-1} follows. 
As for the last claim, we observe that by  \eqref{Va-1} and by \cite[Lemma 3.6(iii)]{fgs_book}, $V$ is $B$-continuous in $X$. 
\end{proof}

We point out that $V$ may not be continuous  with respect to the $|\cdot|_{-1}$ norm in the whole $X$.
%

\section{HJB equation: Viscosity solutions}\label{sec:viscosity}
In this section we characterize $V$ as the unique $B$-continuous viscosity solution to the associated HJB equation. 
Given $v \in C^1(X)$, we  denote by $D v(x)$ its Fr\'echet derivative at $x \in X$ and we write
\begin{align*}
Dv(x)=
  \begin{bmatrix}
    D_{x_0}v(x) \\
    D_{x_1}v(x)
  \end{bmatrix},
\end{align*}
where $D_{x_0}v(x), D_{x_1}v(x)$ are the partial Fr\'echet derivatives.
For $v \in C^2(X)$, we denote by $D^2 v(x)$ its second order Fr\'echet derivative at $x \in X$ which we will often write as
\[
D^2v(x)=
  \begin{bmatrix}
  D^2_{x_0^2}v(x) & D^2_{x_0x_1}v(x)  \\
    D^2_{x_1x_0}v(x)  & D^2_{x_1^2}v(x) 
  \end{bmatrix}.
\]
We define the Hamiltonian function  
 $H:X\times X\times S(X)\to \R$ by
 \begin{small}
\begin{align*}
H(x,p,Z) &:=  \sup_{u\in U} \left \{ \cp{-\tilde b(x,u)}{p}-  {1\over 2}\tr (\sigma(x,u)\sigma(x,u)^*Z) - L(x,u)  \right \} \\
 & =  - x_0 \cdot p_{0}  +  \sup_{u\in U} \Bigg \{ - b_0\left ( x_0,\int_{-d}^0 a_1(\xi)x_1(\xi)\,d\xi ,u \right) \cdot p_0 \\ & -  {1\over 2} \tr \left [ \sigma_0 \left ( x_0,\int_{-d}^0 a_2(\xi)x_1(\xi)\,d\xi ,u \right)\sigma_0 \left ( x_0,\int_{-d}^0 a_2(\xi)x_1(\xi)\,d\xi ,u \right)^T Z_{00}  \right]
 - l(x_0,u)  \Bigg \} \\
  & =  - x_0 \cdot p_0  +  \sup_{u\in U} \Bigg \{ - b_0\left ( x ,u \right) \cdot p_0 -  {1\over 2} \tr \left [ \sigma_0 \left ( x,u \right)\sigma_0 \left ( x,u \right)^T  Z_{00} \right]
 - l(x_0,u)  \Bigg \} \\
 &=:\tilde H \left (x,p_0 ,Z_{00} \right).
\end{align*}
 \end{small}
By \cite[Theorem 3.75]{fgs_book} the Hamiltonian $H$ satisfies the following properties. 
\begin{lemma}\label{lemma:properties_H}
Let Assumptions \ref{hp:state} and \ref{hp:cost} hold.
\begin{enumerate}[(i)]
\item $H$ is uniformly continuous on bounded subsets of $X \times X \times S(X)$. 
\item  For every $x,p \in X$ and every $Y,Z \in S(X)$ such that $Z \leq Y$, we have
\begin{equation}\label{eq:H_decreasing}
H(x,p,Y)\leq H(x,p,Z). 
\end{equation} 
\item For every $x,p \in X$ and every  $R>0$, we have 
\begin{align}\label{eq:H_lambda_BQ_N}
\lim_{N \to \infty} \sup \Big \{   |H(x,p,Z+\lambda BQ_N)-H(x,p,Z)|: \  |Z_{00}|\leq R, \ |\lambda| \leq R \Big \}=0.
\end{align}
\item For every $R>0$ there exists a modulus of continuity $\omega_R$ such that
\begin{align}\label{eq:H_norm_-1}
 H \left (z,\frac{B(z-y)}{\varepsilon},Z \right )-H \left (y,\frac{B(z-y)}{\varepsilon},Y \right )  \geq -\omega_R\left (|z-y|_{-1}\left(1+\frac{|z-y|_{-1}}{\varepsilon}\right) \right) 
 \end{align}
for every $\varepsilon>0$, $y,z \in X$ such that $|y|_X,|z|_X\leq R$, $Y,Z \in \mathcal{S}(X)$ satisfying
$$Y=P_N Y P_N \quad Z=P_N Z P_N$$
and 
\begin{align*}
\frac{3}{\varepsilon}\left(\begin{array}{cc}B P_{N} & 0 \\ 0 & B P_{N}\end{array}\right) \leq\left(\begin{array}{cc}Y & 0 \\ 0 & -Z\end{array}\right) \leq \frac{3}{\varepsilon}\left(\begin{array}{cc}B P_{N} & -B P_{N} \\ -B P_{N} & B P_{N}\end{array}\right).
\end{align*}
\item
 If $C>0$ is the constant in  \eqref{eq:b_sublinear} and \eqref{eq:G_bounded}, then, for every $x \in X, p, q \in X,  Y,Z \in \mathcal{S}(X)$,
\begin{align}\label{eq:Hamiltonian_local_lip}
| H(x, p+q,Y+Z)-H(x, p, Y)| \leq C\left(1+|x|_X\right)|q_0|+\frac{1}{2}C^2\left(1+|x|_X\right)^{2}|Z_{00}|.
\end{align}

\end{enumerate}
\end{lemma}


The Hamilton-Jacobi-Bellman (HJB) equation associated with the optimal control problem is the infinite dimensional PDE
\begin{equation}
\label{eq:HJB}
\rho v(x) - \cp{\tilde A x}{Dv(x)} + H(x,Dv(x),D^2v(x))=0,
\quad x \in X.
\end{equation}
We recall the definition of $B$-continuous viscosity solution from \cite{fgs_book}.
\begin{definition}\label{def:test_functions}
\begin{itemize}
\item[(i)]  $\phi \colon X \to \mathbb R$ is a regular test function if
\begin{small}
\begin{align*}
\phi \in \Phi := \{ \phi \in C^2(X): \phi \textit{ is }B\textit{-lower semicontinuous and }\phi, D\phi , D^2\phi , A^*  D\phi \textit{ are uniformly continuous on }X\};
\end{align*}
\end{small}
\item[(ii)] $g \colon X \to \mathbb R$  is a radial test function if
\begin{align*}
g \in \mathcal G:= \{ g \in C^2(X): g(x)=g_0(|x|_X) \textit{ for some } g_0 \in C^2([0,\infty)) \textit{  non-decreasing}, g_{0}'(0)=0 \}.
\end{align*}
\end{itemize}
\end{definition}
\begin{flushleft}
Note that, if $g\in\mathcal{G}$, we have 
\end{flushleft}
\begin{align}\label{eq:gradient_radial}
D g(x)=\left\{\begin{array}{l}
g_0^{\prime}(|x|_{X}) \frac{x}{|x|_{X}}, \quad \ \ \ \mbox{if} \  x \neq 0, \\
0,  \quad\quad  \ \ \ \ \ \ \ \ \ \ \ \ \,\,\,  \ \mbox{if} \ x=0.
\end{array}\right.
\end{align}
We say that a function is locally bounded if it is bounded on bounded subsets of $X$.
\begin{definition}\label{def:viscosity_solution}
\begin{enumerate}[(i)]
\item A locally bounded $B$-upper semicontinuous function $v:X\to\R$ is a viscosity subsolution of \eqref{eq:HJB} if, whenever $v-\phi-g$ has a local maximum at $x \in X$ for $\phi \in \Phi, g \in \mathcal G$, then 
\begin{equation*}
\rho v(x) - \cp{ x}{\tilde A^* D\phi(x)}_{X} + H(x,D\phi(x)+Dg(x) ,D^2\phi(x)+D^2g(x))\leq 0.
\end{equation*}
\item 
A locally bounded $B$-upper semicontinuous function $v:X\to\R$ is a viscosity supersolution of \eqref{eq:HJB} if, whenever $v+\phi+g$ has a local minimum at $x \in X$ for $\phi \in \Phi$, $g \in \mathcal G$, then 
\begin{equation*}
\rho v(x) + \cp{ x}{\tilde A^* D\phi(x)}_{X} + H(x,-D\phi(x)-Dg(x) ,-D^2\phi(x)-D^2g(x))\geq 0.
\end{equation*}
\item 
A viscosity solution of \eqref{eq:HJB} is a function $v:X\to\R$ which is both a viscosity subsolution and a viscosity supersolution of \eqref{eq:HJB}.
\end{enumerate}
\end{definition}
Define
$$\mathcal{S}:=\{u \colon X \to \mathbb{R}:  \exists k\geq 0 \ \mbox{satisfying \eqref{eq:k_set_uniqueness} and } \tilde C\geq 0 \,\mbox{such that}\,  |u(x)|\leq  \tilde C (1+|x|_X^k)\},$$
where
\begin{small}
\begin{equation}\label{eq:k_set_uniqueness}
\begin{cases}
k<\frac{\rho}{C+\frac{1}{2} C^{2}}, \ \ \ \ \ \ \ \ \ \ \ \ \ \ \ \quad \mbox{ if } \ \frac{\rho}{C+\frac{1}{2} C^{2}} \leq 2, \\
C k+\frac{1}{2} C^{2} k(k-1)<\rho,  \quad \mbox{ if } \frac{\rho}{C+\frac{1}{2} C^{2}}>2, 
\end{cases}
\end{equation}
\end{small}
and $C$ is the constant appearing in \eqref{eq:b_sublinear} and \eqref{eq:G_bounded}.

We can now state the theorem characterizing $V$ as the unique viscosity solution of \eqref{eq:HJB} in $\mathcal{S}$.

\begin{theorem}\label{th:existence_uniqueness_viscosity_infinite}
Let Assumptions \ref{hp:state}, \ref{hp:cost}, and \ref{hp:discount} hold.
The value function $V$ is the unique viscosity solution of \eqref{eq:HJB} in the set $\mathcal{S}$.
\end{theorem}
\begin{proof}
Notice that $V \in \mathcal S$ by Proposition \ref{prop:growth_V}. 

The proof of the fact that $V$ is the unique viscosity solution of the HJB equation can be found in \cite[Theorem 3.75]{fgs_book} as all assumptions of this theorem are satisfied due to Lemma \ref{lemma:properties_H}. The reader can also check
the comparison theorem \cite[Theorem 3.56]{fgs_book}.
\end{proof}
\begin{remark}
We remark that Theorem \ref{th:existence_uniqueness_viscosity_infinite} also holds in the deterministic case, i.e. when $\sigma(x,u)=0$.  (in which case we may take $\rho_0=Cm$ and $k<\rho/C$ in \eqref{eq:k_set_uniqueness}). The theory of viscosity solutions handles well degenerate HJB equations, i.e. when the Hamiltonian satisfies
$$H(x,p,Y) \leq H(x,p,Z) $$
for every $Y,Z \in S(X)$ such that $Z\leq Y$. Hence viscosity solutions can be used in connection with the dynamic programming method for optimal control of stochastic differential equations in the case of degenerate noise in the state equation, in particular, when it completely vanishes (deterministic case). This is not possible using the mild solutions approach due to its various limitations described in the introduction (for more on this, see  \cite{fgs_book}).
 \end{remark}
\section{Partial regularity}\label{sec:regularity}
In this section we prove partial regularity of $V$ with respect to the $x_0$-variable. To do this we assume the following. 
\begin{hypothesis}\label{hp:local_lipschitz_V}
For every $R>0 $ there exists $K_R>0$ such that
$$|V(y)-V(x)|\leq K_R |y-x|_{-1}, \quad \forall x,y \in X, |x|,|y|\leq R.$$
\end{hypothesis}
Assumption \ref{hp:local_lipschitz_V} is satisfied in many natural cases when the cost function  $l(\cdot,u)$ is Lipschitz continuous as we illustrate in the following example.
\begin{example}
Suppose that
 $l(\cdot,u)$ is Lipschitz continuous, uniformly in $u$, and  $\rho >C+\frac{ C^2|B|_{\mathcal{L}(X)}}{2}$, where $C$ is the constant from \eqref{eq:b_B_property} and \eqref{eq:G_lipschitz_norm_B}. Indeed,
 fix $x,y \in X$, $u(\cdot) \in \mathcal{U}$ and denote by $X(s)$, $Y(s)$ the mild solutions of the state equation with initial data $x,y$ respectively and the same control $u(\cdot)$. Then, by \cite[Lemma 3.20]{fgs_book}, we have
\begin{equation}\label{eq:E|X(s)-Y(s)|_-1}
 \mathbb{E}\left[|X(r)-Y(r)|_{-1}^{2}\right] \leq C(r)|x-y|_{-1}^{2},
\end{equation}
where, recalling that the constant of the weak $B$-condition in our case is $C_0=0$,
$$C(r)=e^{(2 C+C^2|B|_{\mathcal{L}(X)})r}.$$
Since $L(x,u)=l(x_0,u)$, the Lipschitz continuity of $l(\cdot,u)$ (uniform in $u$), \eqref{eq:|x_0|_leq_|x|},  \eqref{eq:E|X(s)-Y(s)|_-1} and $\rho > C+\frac{C^2|B|_{\mathcal{L}(X)}}{2}$ yield
\begin{small}
\begin{align*}
& \mathbb{E}  \int_0^{\infty} e^{-\rho r}|L(X(r), u(r))-L(Y(r), u(r))| d r  \leq C_1 \int_{0}^{\infty} e^{-\rho r} \mathbb{E}  |X_0(r)-Y_0(r)| d r		\\
& \leq C_1 \int_{0}^{\infty} e^{-\rho r} \mathbb{E}  |X(r)-Y(r)|_{-1} d r 
 \leq C_1 \int_{0}^{\infty} e^{-\rho r} e^{\left (C+\frac{C^2|B|_{\mathcal{L}(X)}}{2} \right)r} d r |x-y|_{-1}  \leq C_2 |x-y|_{-1}.
\end{align*}
\end{small}
and the claim easily follows.
\end{example}
\begin{remark}
Notice that, if we only assume that $l(\cdot,u)$ is locally Lipschitz continuous, uniformly in $u$, the above proof would not work and, by using an argument outlined in the proof of Proposition \ref{prop:V_continuous_norm_-1}, we would only get that $V$ is uniformly continuous with respect to the $|\cdot|_{-1}$ norm on bounded sets of $X$.
\end{remark}
Finally, we assume local uniform nondegeneracy of $\sigma_0 \colon \mathbb R^n \times L^2 \times U \to \mathbb R^n \times \mathbb R^q$ (recall the definition of $\sigma_0(x,u)$ in Section \ref{sec:infinite_dimensional_framework}).

\begin{hypothesis}\label{hp:uniform_ellipticity}
For every $R>0$  there exists $\lambda_{R} >0$ such that
$$\sigma_0(x,u) \sigma_0(x,u)^T\geq \lambda_{R} I, \quad \forall x\,\ \mbox{such that}\, \ |x|_{X} \leq R,  \ \forall u \in U.$$
\end{hypothesis}
For every $\bar x_1 \in L^2$ we define 
$$V^{\bar x_1}(x_0):=V(x_0,\bar x_1), \quad \forall x_0 \in \mathbb R^n.$$
Theorem \ref{th:C1alpha} is the main result of this section.
\begin{theorem}\label{th:C1alpha}
Let Assumptions \ref{hp:state}, \ref{hp:cost}, \ref{hp:discount}, \ref{hp:local_lipschitz_V}, and \ref{hp:uniform_ellipticity} hold. For every $p>n$ and  every fixed $\bar x_1 \in L^2$, we have $V^{\bar x_1}\in W^{2,p}_{\rm loc}(\mathbb R^n)$; thus,  by Sobolev embedding, $V^{\bar x_1}\in C^{1,\alpha}_{\rm loc}({\mathbb{R}}^n)$ for all $0<\alpha<1$.
Moreover, for every $R>0$, there exists $C_{R}>0$ such that 
$$|V^{\bar x_1}|_{W^{2,p}(B_{R})}\leq C_R, \ \ \ \ \forall \bar x_{1}\ \mbox{such that}\, \ |\bar x_1|_{L^{2}}\leq R.$$ 
Finally, $D_{x_0}V$ is continuous with respect to the $|\cdot |_{-1}$ norm on bounded sets of $X$. In particular, $D_{x_0}V$ is continuous in $X$.
\end{theorem}
Before presenting the proof, we explain its basic idea. Since the range of  $\sigma(x,u)$ is finite dimensional (the $x_1$-component of $\sigma(x,u) w$ for $w \in \mathbb R^q$ is $0$), we first show that the function $V^{\bar x_1}$, defined on a finite dimensional space, is a viscosity subsolution (respectively, supersolution) of a finite-dimensional equation \eqref{eq:proof_V_x1_subsolution} (respectively,   \eqref{eq:proof_V_x1_supersolution}) in $\mathbb R^n$. This part contains the most technical difficulties. Once this is done, we can then apply the theory of $L^p$-viscosity solutions (e.g., \cite{swiech2020}) to obtain that $V^{\bar x_1}$ is an $L^p$-viscosity solution of a linear, uniformly elliptic PDE equation in $\mathbb R^n$ with a locally bounded right hand side function. Then, the regularity theory for uniformly elliptic equations  yields  $V^{\bar x_1}\in W^{2,p}_{\rm loc}(\mathbb R^n)$ for every $p>n$.
\begin{proof}
We organize the proof in several steps.

\underline{\emph{Step 1.}}
Fix $\bar x_1 \in L^2$ and let $R \geq 1$ be  such that $|\bar x_1|_{L^2}\leq R$. Let $\bar x_0 \in \mathbb R^n$ be such that $|\bar x_0|\leq R$ and $\varphi\in C^2({\mathbb{R}^n})$ be such that $ V^{\bar x_1}-\varphi$ has a strict global maximum at $\bar x_0$. We assume without loss of generality that the maximum is equal to $0$ and that $\varphi>0$ if $|x_0|>4R$. We extend $\varphi$ to $X$ by defining $\tilde \varphi(x_0,x_1):= \varphi(x_0)$. With an abuse of notation we will still write $\varphi(x_{0})$ for $\tilde{\varphi}(x)$ and 
$D \varphi(x_0)=(D_{x_0}\varphi(x_0),0)$. Note that  $D \varphi(x_0) \in D(\tilde A^*)$ for all $x \in X$.
   Set $\bar x:=(\bar x_0,\bar x_1)$.
 We consider, for $\varepsilon>0$, the functions
\[
\Phi_\varepsilon(x)=V(x)-\varphi(x_0)-\frac{1}{\varepsilon}|x-(x_0,\bar x_1)|_{-1}^2-M(|x|_{X}-2R)_+^{m+4}, \quad x=(x_0,x_1)\in X,
\]
where $M$ is chosen so that $V(x)<\frac{M}{2}(|x|_{X}-2R)_+^{m+4}$ if $|x|_{X}>4R$. Then
\begin{align}\label{eq:proof_growth_Phi_m}
\Phi_\varepsilon(x)<-\frac{M}{2}(|x|_{X}-2R)_+^{m+4},\quad \forall x\in X, |x|_{X}>4R.
\end{align}
Observe that $\Phi_\varepsilon$ is weakly sequentially upper semicontinuous as $V$ is weakly sequentially continuous by Proposition \ref{prop:V_continuous_norm_-1}, the functions $x\mapsto \varphi(x_0)$ and $x\mapsto |x|_{-1}^2$ are weakly sequentially continuous, and 
the function $x\mapsto |x|_{X}$ is weakly sequentially lowersemicontinuous.

We distinguish two cases: (i) $\sup_{X}\Phi_\varepsilon>0$ for every $\varepsilon>0$; (ii)  $\Phi_\varepsilon \leq 0$ for some (small) $\varepsilon$.

\emph{Case (i)}
 Recall that $V^{\bar x_1}(\bar x_0)-\varphi(\bar x_0)=0$ and $|\bar x|_{X}\leq \sqrt{2}R$. Then
$
\Phi_\varepsilon
$
has global maximum at some $\hat x^{\varepsilon} \in X$, with $|\hat x^{\varepsilon}|_{X}\leq 4R$ and $\Phi_\varepsilon(\hat x^{\varepsilon})>0$. Recalling  Assumption \ref{hp:local_lipschitz_V} and since $V^{\bar x_1}-\varphi$ has a strict global maximum at $\bar x_0$ equal to $0$, we have
\[
\begin{split}
0 < \Phi_\varepsilon(\hat x^{\varepsilon})&
\leq 
V(\hat x^{\varepsilon})-\varphi(\hat x_0^{\varepsilon})-\frac 1 \varepsilon|\hat x-(\hat x_0^{\varepsilon},\bar x_1)|_{-1}^2
\\
&= V(\hat x^{\varepsilon})-V(\hat x_0^{\varepsilon},\bar x_1)+V^{{\bar x}_{1}}(\hat x_0^{\varepsilon})-\varphi(\hat x_0^{\varepsilon})-\frac 1 \varepsilon |\hat x^{\varepsilon}-(\hat x_0^{\varepsilon},\bar x_1)|_{-1}^2
\\
&
\leq
K_{4R}|\hat x^{\varepsilon}-(\hat x_0^{\varepsilon},\bar x_1)|_{-1}-\frac 1 \varepsilon |(0,\hat x_1^{\varepsilon}-\bar x_1)|_{-1}^2=K_{4R}|(0,\hat x_1^{\varepsilon}-\bar x_1)|_{-1}-\frac 1 \varepsilon|(0,\hat x_1^{\varepsilon}-\bar x_1)|_{-1}^2.
\end{split}
\]
It thus follows that
\begin{align}\label{eq:proof_bound_hat_x_1-bar_x_1}
|(0,\hat x_1^{\varepsilon}-\bar x_1)|_{-1}\leq K_{4R}\varepsilon.
\end{align}
On the other hand, by the fact that $\sup_{X} \Phi_\varepsilon >0$, we have
\begin{align*}
V^{\bar x_1}(\bar x_0)-\varphi(\bar x_0)& =\sup_{\R^{n}} (V^{\bar x_1}-\varphi)=0 < \sup_{X} \Phi_\varepsilon= \Phi_\varepsilon(\hat x^{\varepsilon})\\
& \leq V(\hat x^{\varepsilon})-\varphi(\hat x_0^{\varepsilon})= V(\hat x^{\varepsilon})-V^{\bar x_1}(\hat x_0^{\varepsilon})+V^{\bar x_1}(\hat x_0^{\varepsilon})-\varphi(\hat x_0^{\varepsilon}).
\end{align*}
Now,  taking the $\liminf_{\varepsilon \to 0}$ above, by \eqref{eq:proof_bound_hat_x_1-bar_x_1} and Assumption \ref{hp:local_lipschitz_V}, we obtain
  $$V^{\bar x_1}(\bar x_0)-\varphi(\bar x_0) \leq \liminf_{\varepsilon \to 0} (V^{\bar x_1}(\hat x_0^{\varepsilon})-\varphi(\hat x_0^{\varepsilon})).
  $$ 
Since $V^{\bar x_1}-\varphi$ has a strict global maximum at $\bar x_0$ we thus must have 
  \begin{align}\label{eq:proof_convergence_hat_x_0-bar_x_0}
\lim_{\varepsilon \to 0}|\hat x_0^{\varepsilon}-\bar x_0|=0.
\end{align}
In particular, from \eqref{eq:proof_bound_hat_x_1-bar_x_1} and \eqref{eq:proof_convergence_hat_x_0-bar_x_0} it now follows that
\begin{align}\label{eq:proof_regularity_convergence_hat_x-bar_x_-1}
\lim_{\varepsilon \to 0}|\hat x^\varepsilon-\bar x|_{-1}=0.
\end{align}

\emph{Case (ii)} In this case $\Phi_\varepsilon$ has a maximum at $\bar x$ since $\Phi_\varepsilon(\bar x)=0$, so we easily get \eqref{eq:proof_regularity_convergence_hat_x-bar_x_-1} with $\hat{x}^\varepsilon=\bar x$.

\medskip
\underline{\emph{Step 2.}}
Define $\psi=\phi+g$, where
\begin{align*}
  \phi(x):=\varphi(x_0)+\frac 1 \varepsilon |x-(x_0,\bar x_1)|_{-1}^2, \ \ \  \ \ \ \ 
g(x):=M(|x|-2R)_+^{m+4}.
\end{align*} 
With these definitions we have $$V-\phi-g=\Phi_{\varepsilon},$$
so that $V-\phi-g$ has a global maximum at $\hat x^\varepsilon$. 
Moreover,
\begin{align*}
&  D\phi(x)=D\varphi(x_0)+ \frac 2 \varepsilon B(x-(x_0,\bar x_1)),\\
&D_{x_0}\phi(x)=D_{x_0}\varphi(x_0),
\end{align*} 
and
\begin{align}\label{eq:proof_growth_Dg,D^2g}
|D^2g(x)|+|Dg(x)|\leq \bar C_R, \quad \forall  |x|\leq 4R.
\end{align}
Notice that, since $D\varphi(x_0) \in D(\tilde A^*)$, we have $\phi\in\Phi$, i.e., it is  is a regular test function  according to Definition \ref{def:test_functions}. Moreover, clearly  $g\in\mathcal{G}$, i.e., it  is a radial test function according to Definition \ref{def:test_functions}.
We will write $\bar C_R$ to denote a generic constant, possibly changing from line to line, depending on $R$ and the data of the problem, which is independent of $\varepsilon$, $\varphi$ and on ${x} \in B_R$. Then, since $V$ is a viscosity subsolution to \eqref{eq:HJB}, we have
\[
\begin{split}
\rho V(\hat x^\varepsilon)&  -\langle  \hat x^\varepsilon,\tilde{A}^*(D\varphi(\hat x_0^\varepsilon)+\frac{1}{\varepsilon}B(0,\hat x_1^\varepsilon-\bar x_1))\rangle +H(\hat x^\varepsilon,D\psi(\hat x^\varepsilon), D^2\psi(\hat x^\varepsilon))\leq 0.
\end{split}
\]
By \eqref{eq:A*B^1/2} and \eqref{eq:proof_bound_hat_x_1-bar_x_1}, we have 
$$\frac 1 \varepsilon|\tilde{A}^*B(0,\hat x_1^\varepsilon-\bar x_1))|=\frac 1 \varepsilon |\tilde{A}^*B^{1/2}B^{1/2}(0,\hat x_1^\varepsilon-\bar x_1))|\leq  |\tilde{A}^*B^{1/2}|_{\mathcal{L}(X)}\frac 1 \varepsilon|(0,\hat x_1^\varepsilon-\bar x_1))|_{-1}\leq \bar C_R.$$  
The latter two inequalities, Proposition \ref{prop:growth_V}, \eqref{eq:Hamiltonian_local_lip}, \eqref{eq:proof_growth_Dg,D^2g}, the fact that $|\hat x^\varepsilon|_X\leq 4R$, and the definition of $\tilde H$ imply
\[
-\langle  \hat x^\varepsilon,\tilde{A}^*D\varphi(\hat x_0^\varepsilon)\rangle+\tilde H(\hat x^\varepsilon,D_{x_0}\varphi(\hat x_0^\varepsilon), D_{x_0^2}^2\varphi(\hat x_0^\varepsilon)) \leq \bar C_R.
\]
Recalling \eqref{eq:proof_convergence_hat_x_0-bar_x_0}, we have
\[
\lim_{\varepsilon \to 0}\tilde{A}^*D\varphi(\hat x_0^\varepsilon)= \tilde{A}^*D\varphi(\bar x_0).
\]
Hence, letting $\varepsilon \to 0$ in the previous inequality and using \eqref{eq:proof_regularity_convergence_hat_x-bar_x_-1} and the continuity of $\tilde H$, we obtain 
\[
-\langle  \bar x,\tilde{A}^*D\varphi(\bar x_0)\rangle+\tilde H(\bar x,D_{x_0}\varphi(\bar x), D^2_{x_0^2}\varphi(\bar x)) \leq \bar C_R.
\]
Now, since $V^{\bar x_1}-\varphi$ has a strict global maximum at $\bar x_0$  and by Assumption \ref{hp:local_lipschitz_V} the function $V^{\bar x_1}$ is Lipschitz continuous on ${\mathbb{R}^n}$, with Lipschitz constant $\bar C_R$ independent of $\bar x_1$, we have 
$|D_{x_0}\varphi(\bar x_0)| \leq \bar C_R$.  Then, as $D\varphi(\bar x_0) \in D(\tilde A^*)$ we have $|\tilde A^*D\varphi(\bar x_0)|_X=|-(D_{x_0}\varphi(\bar x_0) ,0) |_X \leq \bar C_R$.

Therefore, by \eqref{eq:Hamiltonian_local_lip}, $|\bar x|\leq \sqrt{2}R$ and the definition of $\tilde H$, we have  
\begin{align*}
\tilde H(\bar x,0,D_{x_0^2}^2\varphi(\bar x_0)) - \bar C_R \leq  \tilde H(\bar x,D_{x_0}\varphi(\bar x_0),D^2_{x_0^2}\varphi(\bar x_0)),
\end{align*}
so that we obtain
\begin{align*}
\tilde H(\bar x,0,D_{x_0^2}^2\varphi(\bar x_0))
\leq \bar C_R, 
\end{align*}
i.e., 
\begin{align}\label{eq:aaaa2}
\sup_{u \in U} \left [-\frac{1}{2}{\rm Tr}(\sigma_0(\bar x,u)\sigma_0(\bar x,u)^TD_{x_0^2}^2\varphi(\bar x_0)) \right]
\leq \bar C_R
\end{align}
for some constant $\bar C_R>0$ independent of $\varphi$ and $\bar x$ if $|\bar x_0|\leq R, |\bar x_1|\leq R$.
 Thus, for every $\bar x_1$ with $|\bar x_1|\leq R$, the function $V^{\bar x_1}$ is a viscosity subsolution of the finite dimensional equation
\begin{align}\label{eq:proof_V_x1_subsolution}
\sup_{u \in U} \left [-\frac{1}{2}{\rm Tr}(\sigma_0((x_0,\bar x_1),u)\sigma_0((x_0,\bar x_1),u)^TD^2v(x_{0})) \right]
=\bar C_R, \quad \mbox{in}\,\,\{|x_0| < R\}.
\end{align}
Similarly, we prove that  for every $|\bar x_1|\leq R$, the function $V^{\bar x_1}$ is a viscosity supersolution of
\begin{align}\label{eq:proof_V_x1_supersolution}
\sup_{u \in U} \left [-\frac{1}{2}{\rm Tr}(\sigma_0((x_0,\bar x_1),u)\sigma_0((x_0,\bar x_1),u)^TD^2v(x_{0})) \right]
=-\bar C_R, \quad \mbox{in}\,\,\{|x_0| < R\}.
\end{align}
\underline{\emph{Step 3.}} We now employ 
the theory of $L^p$-viscosity solutions. Since the readers may not be familiar with it, we will proceed slowly. Hypothesis \ref{hp:uniform_ellipticity} guarantees that the Bellman operator appearing in \eqref{eq:proof_V_x1_subsolution} and \eqref{eq:proof_V_x1_supersolution} is uniformly elliptic in $\{|x_0| < R\}$. Thus, by \cite[Proposition 2.9]{caffarelli_c_s}, for every $p>n$ the function $V^{\bar x_1}$ is an  $L^p$-viscosity   subsolution of \eqref{eq:proof_V_x1_subsolution} and an $L^p$-viscosity supersolution of \eqref{eq:proof_V_x1_supersolution}. 
Now, by \cite[Proposition 3.5]{caffarelli_c_s}, the function $V^{\bar x_1}$ is twice pointwise differentiable a.e. in $\{|x_0| < R\}$ and by \cite[Proposition 3.4]{caffarelli_c_s} we have that \eqref{eq:proof_V_x1_subsolution} and \eqref{eq:proof_V_x1_supersolution} are satisfied pointwise a.e.
Then, defining the function
$$f(x_0):=\sup_{u \in U} \left [-\frac{1}{2}{\rm Tr}(\sigma_0((x_0,\bar x_1),u)\sigma_0((x_0,\bar x_1),u)^TD^2V^{\bar x_1}(x_0)) \right],$$
we have $|f|_{L^\infty(B_{R})}\leq \bar C_R$ (measurability of $f$ is explained for instance in \cite{swiech2020}).
We can then apply \cite[Corollary 3]{swiech2020} (first such result was stated for $L^p$-viscosity solutions in \cite{swiech1997}) to get that for every $p>n$, the function $V^{\bar x_1}$ is an $L^p$-viscosity solution of
\begin{align}\label{eq:proof_V_x1_L^p_solution_f}
\sup_{u \in U} \left [ -\frac{1}{2}{\rm Tr}(\sigma_0((x_0,\bar x_1),u)\sigma_0((x_0,\bar x_1),u)^T D^2v(x_{0}) \right]
=f(x_0),\quad \mbox{in}\,\,\{|x_0| < R\}.
\end{align}
We now conclude by standard elliptic regularity (see e.g. \cite[Theorem 7.1]{CC}, together with Remark 1 there, or  \cite[Theorem 3.1]{swiech1997}) that $V^{\bar x_1}\in W^{2,p}_{\rm loc}(B_R)$ and $|V^{\bar x_1}|_{W^{2,p}(B_{R/2})}\leq C_R$ for some constant $C_R$. Thus in particular, by Sobolev embeddings, $V^{\bar x_1}\in C^{1,\alpha}_{\rm loc}(B_R)$ for all $0<\alpha<1$.

\underline{\emph{Step 4.}}
We now prove that $D_{x_0}V$ is continuous in $|\cdot|_{-1}$ norm on bounded sets of $X$.\\
Let $|x|_{X}\leq R$ and assume without loss of generality that $x=(0,x_1)$ and $V(0,x_1)=0,D_{x_0}V(0,x_1)=0$. Suppose by contradiction that there is $\varepsilon>0$ and a sequence 
 $x^N=(x^N_0,x^N_1)\to x=(0,x_1)$ in $|\cdot|_{-1}$ norm such that $(x^N_0,x^N_1)\in B_{2R}$ and $|(D_{x_0}V(x^N_0,x^N_1),0)|_{-1}\geq \varepsilon$. We remind that the $|(\cdot,0)|_{-1}$ and the standard norm in $\mathbb R^n$ are equivalent. Let $K_R$ be from Assumption \ref{hp:local_lipschitz_V}. Observe that, since $V^{ x_1}\in C^{1,\alpha}_{\rm loc}({\mathbb{R}}^n)$, for every $y_0 \in \mathbb R^n$ such that 
\begin{equation}\label{eq:proof_norm_(y_0,0)_-1}
|(y_0,0)|_{-1}\leq(K_{3R}|(0,x^N_1-x_1)|_{-1})^{\frac{1}{1+\alpha}}+2|(x_0^N,0)|_{-1},
\end{equation}
we have
 \begin{align*}
 |V(y_0,x_1)|=|V(y_0,x_1)-V(0,x_1)|\leq C |(y_0,0)|_{-1}^{1+\alpha}\leq C(K_{3R}|(0,x^N_1-x_1)|_{-1}+|(x_0^N,0)|_{-1}^{1+\alpha}).
 \end{align*}
Then 
 \begin{align}\label{eq:proof_growth_V(y_0,x_1N)}
|V(y_0,x_1^N)|\leq |V(y_0,x_1)|+K_{3R}|(0,x^N_1-x_1)|_{-1}\leq C(K_{3R}|(0,x^N_1-x_1)|_{-1}+|(x_0^N,0)|_{-1}^{1+\alpha}).
 \end{align}
Now observe that by taking
\[
y_0=x_0^N+\frac{D_{x_0}V(x^N_0,x^N_1)}{|(D_{x_0}V(x^N_0,x^N_1),0)|_{-1}}((K_{3R}|(0,x^N_1-x_1)|_{-1})^{\frac{1}{1+\alpha}}+|(x_0^N,0)|_{-1})
\]
 we have \eqref{eq:proof_norm_(y_0,0)_-1} so that \eqref{eq:proof_growth_V(y_0,x_1N)} holds for $|V(y_0,x_1^N)|$.
 Moreover note that also $x_0^N$ satisfies \eqref{eq:proof_norm_(y_0,0)_-1} so that  we have \eqref{eq:proof_growth_V(y_0,x_1N)} for $|V(x_0^N,x_1^N)|$.
Now, since $V^{ x_1^N}\in C^{1,\alpha}_{\rm loc}({\mathbb{R}}^n)$, we have 
\[
\begin{split}
V(y_0,x_1^N)&\geq V(x_0^N,x_1^N)+D_{x_0}V(x^N_0,x^N_1) \cdot (y_0-x_0^N)-C|y_0-x_0^N|^{1+\alpha}\geq V(x_0^N,x_1^N)
\\
&+\frac{1}{|(D_{x_0}V(x^N_0,x^N_1),0)|_{-1}}|D_{x_0}V(x^N_0,x^N_1)|^2 (K_{3R}|(0,x^N_1-x_1)|_{-1})^{\frac{1}{1+\alpha}}+|(x_0^N,0)|_{-1})
\\
&-C((K_{3R}|(0,x^N_1-x_1)|_{-1})^{\frac{1}{1+\alpha}}+|(x_0^N,0)|_{-1})^{1+\alpha}.
\end{split}
\]
Therefore, using $|(D_{x_0}V(x^N_0,x^N_1),0)|_{-1}\geq \varepsilon$, the fact that \eqref{eq:proof_growth_V(y_0,x_1N)} holds for $|V(x_0^N,x_1^N)|$, and since $|(\cdot,0)|_{-1}$ is an equivalent norm in $\mathbb R^n$, we obtain
\begin{align*}
V(y_0,x_1^N) \geq \eta\varepsilon((K_{3R}|x^N_1-x_1|_{-1})^{\frac{1}{1+\alpha}}+|(x_0^N,0)|_{-1})-C(K_{3R}|(0,x^N_1-x_1)|_{-1}+|(x_0^N,0)|_{-1}^{1+\alpha}),
\end{align*}
where $\eta>0$ is a constant.
Finally, since \eqref{eq:proof_growth_V(y_0,x_1N)} holds for $|V(y_0,x_1^N)|$, we have
\[
\eta\varepsilon((K_{3R}|(0,x^N_1-x_1)|_{-1})^{\frac{1}{1+\alpha}}+|(x_0^N,0)|_{-1})\leq C(K_{3R}|(0,x^N_1-x_1)|_{-1}+|(x_0^N,0)|_{-1}^{1+\alpha})
\]
which is impossible for large $N$ as $\alpha>0$. This concludes the proof of the theorem.
\end{proof}
The regularity result interesting on its own. It possibly can also be used to define an optimal feedback map under some natural assumptions.
 
 Assume that $U$ is compact and that 
 $\sigma_0$ does not depend on $u$.
The Hamiltonian then has the form
\begin{small}
\begin{align*}
H(x,Dv(x),D^2v(x)) & = H(x,Dv(x)) - x_0 \cdot D_{x_0}v(x)   -  {1\over 2} \tr \left [ \sigma_0 \left ( x\right) \sigma_0 \left ( x \right)^T  D^2_{x^2_0}v(x) \right]
\nonumber \\
 &=\tilde H \left (x,D_{x_0}v(x) \right) - x_0 \cdot D_{x_0}v(x) \nonumber\\
 &  -  {1\over 2} \tr \left [ \sigma_0 \left ( x_0,\int_{-d}^0 a_2(\xi)x_1(\xi)\,d\xi \right)\sigma_0 \left ( x_0,\int_{-d}^0 a_2(\xi)x_1(\xi)\,d\xi\right)^T  D^2_{x^2_0}v(x)  \right],
\end{align*}
\end{small}
where 
\begin{small}
$$ H(x,Dv(x))=\tilde H \left (x,D_{x_0}v(x) \right)= \max_{u\in U} \Bigg \{ - b_0\left ( x_0,\int_{-d}^0 a_1(\xi)x_1(\xi)\,d\xi ,u \right) \cdot D_{x_0}v(x) 
 - l(x_0,u)  \Bigg \}.$$
 \end{small}
By Theorem \ref{th:C1alpha} we can define a candidate optimal feedback map, i.e. 
 $$u^*(x)\in \mbox{argmax}_{u\in U} \Bigg \{ - b_0\left ( x_{0},\int_{-d}^0 a_1(\xi)x_1(\xi)\,d\xi ,u \right) \cdot D_{x_0}V(x) 
 - l(x,u)  \Bigg \}. $$
To show that this may lead to the existence of an optimal feedback control is a difficult problem, which passes through a verification theorem with only partial regularity and the study of the closed loop equation (see, e.g. \cite{fgs_book, federico_gozzi_2018} in the context of the approach via mild solutions or \cite{goldys_2} in the case of optimal control of deterministic delay equations). It will likely require additional assumptions. We plan to investigate this in a future work.

\section{Applications}\label{sec:merton}
In this section we provide two examples of possible applications of our approach. The first arises in finance, the second in marketing.
\subsection{Merton-like problem with path dependent coefficients}
We consider a financial market composed by a risk-free asset (bond) $B$ and a risky asset (stock) $S$. The respective dynamics (deterministic for the bond, stochastic for the stock) are given by  
\begin{align*}
 & db(t)=rb(t) dt, \\
& ds(t)=\mu \left (\int_{-d}^0 a_1(\xi) s(t+\xi)  d \xi \right) s(t)dt + v \left (\int_{-d}^0 a_2(\xi) s(t+\xi)  d \xi \right) s(t)dW(t),
\end{align*}
with initial data $s(0)=s_{0}>0, s(\xi)=s_1(\xi)>0$ for every $\xi \in [-d,0]$, $b(0)=1$, and $W$ is a real-valued  Brownian motion. Moreover, 
\begin{itemize}
\item[(i)] $r\geq 0$;
\item[(ii)] $a_1,a_{2}$ and are given deterministic functions satisfying the assumptions used in the previous sections;
\item[(iii)] $\mu,v:\R\to\R$ are given Lipschitz continuous functions. 
\end{itemize}
The investor chooses a consumption-investment strategy by deciding  at time $t\geq 0$ the fraction  $u(t) \in [0,1]$ of the portfolio $z(t)$ to be invested in the risky stock $S$; the remaining part $1-u(t)$ is then invested in the bond $B$ (self-financed portfolio with no borrowing and no short selling constraints); 
The dynamic of the portfolio (wealth) $z(t)$ is then
\begin{align*}
 &dz(t) =\frac{ds(t)}{s(t)}u(t)z(t)+\frac{db(t)}{b(t)}(1-u(t))z(t)   dt\\
&=\left[rz(t) + \left[\mu \left (\int_{-d}^0 a_1(\xi) s(t+\xi)d\xi \right) -r\right]u(t)z(t)\right]dt +  \nu \left (\int_{-d}^0 a_2(\xi) s(t+\xi) d \xi \right) u(t) z(t) dW(t),
\end{align*}
with $z(0)=z_0$, where $z_{0}>0$ is the initial value of the portfolio. The stochastic process $u(\cdot)$ is the control process. We use the same setup of the stochastic optimal control problem as the one in Section \ref{sec:formul}. The control set $U$ is now 
\[
U= [0,1] .
\]  
The state equation
 of the optimal control problem can be seen as a controlled SDDE in $\R^{2}$ for the couple $y(t)=(s(t),z(t))$ in the form \eqref{eq:SDDE}, 
where, for every $x_0=(s_0,z_0) \in \mathbb{R}^2,\ (s_1,z_1) \in L^2,\ u \in U,$ 

\begin{align*}
b_0 \left ( x_0,\int_{-d}^0 a_1(\xi)x_1(\xi)\,d\xi ,u \right)=\begin{bmatrix}
&  \mu \left (\int_{-d}^0 a_1(\xi)s_1(\xi) d \xi \right) s_0 \\
& r z_0 + \left[ \mu \left (\int_{-d}^0 a_1(\xi)s_1(\xi) d \xi \right)  -r \right]u z_0,  
 \end{bmatrix},
\end{align*}

\begin{align*}
\sigma_0 \left (x_0,\int_{-d}^0 a_2(\xi)x_1(\xi)\,d\xi ,\, u \right)w=\begin{bmatrix}
&  \nu \left (\int_{-d}^0 a_2(\xi)s_1(\xi) d \xi \right) s_0 w\\
&  \nu \left (\int_{-d}^0 a_2(\xi)s_1(\xi) d \xi \right) u z_0 w
 \end{bmatrix}. 
\end{align*}
The
goal of the investor is to solve the following optimization problem:
\begin{equation*}
\sup_{u(\cdot) \in \mathcal{U}} \mathbb E\left[ \int_0^{\infty} e^{-\rho t}g(z(t))dt \right]
\end{equation*}
for some concave utility function $g \colon \mathbb R \rightarrow \mathbb R$, where $\rho>0$ is a discount factor. The optimization of this kind of functionals arises in mathematical finance, for example in the context of portfolio optimization with random horizon  (see e.g. \cite[Section 6.1]{federico_gassiat_gozzi_2015}) or in the context of pension fund management  (see e.g. \cite{dizacinto_federico_gozzi_2011, federico_2011}).

Note that the maximization problem is equivalent to 
\begin{equation*}
\inf_{u(\cdot) \in \mathcal{U}} \mathbb E\left[ \int_0^{\infty} e^{-\rho t}l(z(t))dt \right],
\end{equation*}
where $l(z)=-g(z)$. 
By considering the infinite dimensional framework of Section \ref{sec:infinite_dimensional_framework}, if $l$ (or equivalently $g$) satisfies Assumption \ref{hp:cost}, we can use Theorem \ref{th:existence_uniqueness_viscosity_infinite} to characterize the value function $V$ as the unique viscosity solution to \eqref{eq:HJB}.

\subsection{Optimal advertising with delays}
The following problem is taken from \cite{gozzi_marinelli_2004}. In the spirit of the model in \cite[Section 4]{gozzi_marinelli_2004} we assume that no delay in the control is present.
The model for the dynamics of the stock of advertising goodwill $y(s)$ of the product is given by the following controlled SDDE
\begin{equation*}
\begin{cases}
dy(t) = \left[ a_0 y(t)+\int_{-d}^0 a_1(\xi)y(t+\xi)\,d\xi +c_0 u(t) \right]
         dt   + \sigma_0 \, dW(t),\\
y(0)=x_0, \quad y(\xi)=x_1(\xi)\; \quad \forall \xi\in[-d,0),
\end{cases}
\end{equation*}
where $d>0$, the control process $u(s)$ models the intensity of advertising spending and $W$ is a real-valued Brownian motion.
\begin{enumerate}[(i)]
\item $a_0 \leq 0$ is a constant factor of image deterioration in absence of advertising;
\item   $c_0 \geq 0$ is a constant advertising effectiveness factor;
\item $a_1 \leq 0$ is a given deterministic function satisfying the assumptions used in the previous sections which represents the distribution of the forgetting time; 
\item $\sigma_0>0$ represents  the uncertainty in the model;
\item $x_0 \in \mathbb R$ is the level of goodwill at the beginning of the advertising campaign;
\item $x_1 \in L^2([-d,0];\mathbb R)$ is the history of the goodwill level.
\end{enumerate}
Again, we use the same setup of the stochastic optimal control problem as the one in Section \ref{sec:formul} and the control set $U$ is here
\[
U= [0,\bar u]
\]  
for some $\bar u>0$.
The optimization problem is 
$$
\inf_{u \in \mathcal U}\mathbb{E} \left[\int_0^\infty e^{-\rho s} l(y(s),u(s)) d s\right],
$$
where  $\rho >0$ is a discount factor, $l(x,u)=h(u)-g(x)$, with  a continuous and convex cost function $h \colon U \rightarrow \mathbb R$ and a continuous and concave utility function $g \colon \mathbb R \rightarrow \mathbb R$ which satisfies Assumption \ref{hp:cost}.

Setting
$$b_0 \left ( x_0,\int_{-d}^0 a_1(\xi)x_1(\xi)\,d\xi ,u\right):=a_0 x_0+\int_{-d}^0 a_1(\xi)x_1(\xi)\,d\xi +c_0 u,$$ 
we are then  in the setting of Section \ref{sec:formul}. Therefore, using the infinite dimensional framework of Section \ref{sec:infinite_dimensional_framework}, we can use Theorem \ref{th:existence_uniqueness_viscosity_infinite} to characterize the value function $V$ as the unique viscosity solution to \eqref{eq:HJB}, and Theorem \ref{th:C1alpha} to obtain partial regularity of $V$.

\bibliographystyle{amsplain}

\end{document}